\documentclass[12pt]{article}
\usepackage[a4paper,includeheadfoot,margin=2.54cm]{geometry}
\usepackage[english]{babel}
\usepackage{amsthm}
\usepackage{amsmath}
\usepackage{amssymb}
\usepackage[hyphens]{url}

\usepackage{multirow,graphicx}
\usepackage[sort]{cite}

\usepackage[all]{xy}
\usepackage{xcolor}

\theoremstyle{plain}
\newtheorem{lemma}{Lemma}
\newtheorem{proposition}[lemma]{Proposition}
\newtheorem{theorem}[lemma]{Theorem} 
\newtheorem{corollary}[lemma]{Corollary}
\newtheorem{example}[lemma]{Example} 
\newtheorem{remark}[lemma]{Remark}

\newcommand{\F}{\mathbb{F}}

\newcommand{\V}{\mathbb{V}}

\renewcommand{\to}{\rightarrow}

\newcommand{\Tr}{{\rm Tr}}

\begin{document}

\title{On generalizations of Maiorana--McFarland and $\mathcal{PS}_{ap}$ functions\\
	(Sur des généralisations des fonctions de Maiorana--McFarland et $\mathcal{PS}_{ap}$)
	}

\author{ 
	Sezel Alkan$^{1}$, Nurdag\"{u}l Anbar$^{1}$, Athina Avrantini$^{2}$,\\Erroxe Etxabarri-Alberdi$^{3}$, Tekg\"{u}l Kalayc\i$^{4}$, Beatrice Toesca$^{5}$
	\vspace{0.4cm} \\
	\small $^1$ Sabanc{\i} University, MDBF, Orhanl\i, Tuzla, 34956 \. Istanbul, Turkey\\
	\small $^2$ Department of Mathematics, University of Pennsylvania, Philadelphia, PA 19104-6395, USA
	\\
	\small $^3$ Basic Sciences Department,  Mondragon Unibertsitatea, 20500 Arrasate, Gipuzkoa, Spain
	\\
	\small $^4$ Institut f\"ur Mathematik, Alpen-Adria-Universit\"at Klagenfurt, Austria \\
	\small $^5$ Institute of Mathematics, University of Zurich, Switzerland
	\\
	\small Email: {\tt sezel.alkan@sabanciuniv.edu } \\
	\small Email: {\tt nurdagulanbar2@gmail.com}\\
	\small Email: {\tt athina@sas.upenn.edu }\\
	\small Email: {\tt eetxabarri@mondragon.edu  } \\
	\small Email: {\tt tekgulkalayci1@gmail.com}\\
	\small Email: {\tt beatrice.toesca@math.uzh.ch  }
}
\date{}

	\maketitle

	\begin{abstract}
		We investigate generalizations of two classical
		constructions of Boolean bent functions, namely
		the Maiorana--McFarland ($\mathcal{M}$) class
		and the Desarguesian partial spread
		($\mathcal{PS}_{ap}$) class.
		
		Motivated by the growing interest in finding bent
		functions outside the completed $\mathcal{M}$ class,
		we construct new families of generalized
		Maiorana--McFarland bent functions that are not
		EA-equivalent to any function in the classical
		$\mathcal{M}$ or $\mathcal{PS}_{ap}$ classes,
		and hence lie outside their respective completed
		classes.
		
		We also investigate the decomposition properties
		of generalized $\mathcal{PS}_{ap}$ bent functions
		and establish sufficient conditions under which
		they admit neither bent nor semibent decompositions.
		In particular, we obtain indecomposability results
		for bent functions associated with permutation
		polynomials, including almost perfect nonlinear
		(APN) permutation monomials. These results yield
		further families of classical $\mathcal{PS}_{ap}$
		bent functions admitting neither bent nor semibent
		decompositions, thereby providing a partial answer
		to a previously posed open problem.
		
		Our approach relies on properties of second-order
		derivatives of bent functions and their duals,
		combined with techniques from the theory of
		algebraic curves and function fields.
	\end{abstract}

	
	
\section{Introduction}

Let $\V_n^{(2)}$ be an $n$-dimensional vector space over $\F_2$, equipped with a fixed nondegenerate inner product $\langle\cdot,\cdot\rangle_n$. For $\V_n^{(2)}=\F_2^n$, we use the usual dot product. When $\V_n^{(2)}=\F_{2^n}$, we put $\langle b,x\rangle_n=\Tr_1^n(bx)$, where $\Tr_k^n$ denotes the trace from $\F_{2^n}$ to $\F_{2^k}$ for $k\mid n$. More generally, for the representation
$\V_n^{(2)}=\F_{2^{n_1}}\times\cdots\times\F_{2^{n_r}}$
with $n=n_1+\cdots+n_r$, we use the inner product
$\langle b,x\rangle_n=\sum_{i=1}^{r}\Tr_1^{n_i}(b_i x_i)$,
where $b=(b_1,\ldots,b_r)$ and $x=(x_1,\ldots,x_r)$.

A \emph{Boolean function} is a mapping $f:\V_n^{(2)}\to\F_2$. Its \emph{Walsh transform} is
\[
W_f(b)=\sum_{x\in\V_n^{(2)}}(-1)^{f(x)+\langle b,x\rangle_n},\qquad b\in\V_n^{(2)}.
\]
The \emph{extended Walsh spectrum} of $f$ is the multiset $\{|W_f(b)|:b\in\V_n^{(2)}\}$. The function $f$ is called \emph{bent} if $|W_f(b)|=2^{n/2}$ for every $b\in\V_n^{(2)}$. Equivalently, each nonzero first-order derivative $D_af(x)=f(x+a)+f(x)$ is balanced. Boolean bent functions exist only for even $n$. For a bent function $f$, there is a unique Boolean bent function $f^*$, called its \emph{dual}, satisfying
\[
W_f(b)=2^{n/2}(-1)^{f^*(b)},\qquad b\in\V_n^{(2)}.
\]

A Boolean function $f$ is called \emph{$s$-plateaued} if
\[
|W_f(b)|\in\{0,2^{(n+s)/2}\}\quad\text{for every }b\in\V_n^{(2)},
\]
where $0\le s\le n$ and $n+s$ is even. Thus $0$-plateaued functions are precisely bent functions. A Boolean function is called \emph{semibent} if it is $1$-plateaued when $n$ is odd, or $2$-plateaued when $n$ is even.

Let $F:\V_n^{(2)}\to\V_k^{(2)}$ be a function. For each nonzero $\alpha\in\V_k^{(2)}$, its \emph{component function} is $F_\alpha(x)=\langle\alpha,F(x)\rangle_k$. The function $F$ is \emph{vectorial bent} if every nonzero component is bent. Equivalently, for each nonzero $a\in\V_n^{(2)}$, the derivative $D_aF(x)=F(x+a)+F(x)$ is balanced as a mapping into $\V_k^{(2)}$.

Two functions $F,G:\V_n^{(2)}\to\V_k^{(2)}$ are \emph{extended-affine equivalent (EA-equivalent)} if
\[
G(x)=L_1(F(L_2(x)+a))+L_3(x)+b,
\]
where $L_1$ and $L_2$ are linear permutations on $\V_k^{(2)}$ and $\V_n^{(2)}$, respectively, $L_3:\V_n^{(2)}\to\V_k^{(2)}$ is linear, and $a\in\V_n^{(2)}$, $b\in\V_k^{(2)}$. For Boolean functions, this becomes $g(x)=f(L(x)+a)+\langle c,x\rangle_n+b$, where $L$ is a linear permutation on $\V_n^{(2)}$, $a,c \in\V_n^{(2)} $, and $b\in \F_2$. EA-equivalence preserves the extended Walsh spectrum and, consequently, bentness. A class $\mathcal C$ of Boolean bent functions is called \emph{completed} if it is closed under EA-equivalence;
its \emph{completion} $\mathcal C^\#$ is the smallest such class containing $\mathcal C$.

Bent functions may be constructed by \emph{primary constructions}, which directly specify bent functions, or by \emph{secondary constructions}, which use previously known bent or related functions. 
Numerous secondary constructions of Boolean bent functions
have been developed in the literature
(see~\cite[Section~6.1.16]{Carlet2021}).
In contrast, primary constructions are relatively few,
with two of the most extensively studied examples being
the Maiorana--McFarland class $\mathcal{M}$~\cite{mm}
and the partial spread class $\mathcal{PS}$~\cite{dillon},
both introduced in the 1970s.
These two fundamental classes, together with their
various generalizations, have played a central role
in the development of the theory of bent functions.

It is well known that every bent function in at most six variables belongs to the completed Maiorana--McFarland class $\mathcal M^\#$. Dillon~\cite{dillon} constructed the first examples outside $\mathcal M^\#$ in eight variables. Computer enumeration subsequently showed that $\mathcal M^\#$ represents only a small fraction of all eight-variable bent functions: its size is at most $2^{72}$, whereas the total number of bent functions is approximately $2^{106}$; see~\cite{ll}. This motivates the search for constructions outside $\mathcal M^\#$; see~\cite{Letal2025,PasalicEtAl2024,PasalicEtAl2023,pp,Pasalic26}.

Membership in $\mathcal M^\#$ can be characterized by Dillon's criterion involving second-order derivatives~\cite{dillon}. For the completed Desarguesian partial spread class $\mathcal{PS}_{ap}^\#$, algebraic degree and $2$-rank provide useful obstructions to membership; see~\cite{akm23,wfq}.

The remainder of the paper is organized as follows.
In Section~\ref{sec:MP}, we recall the classical
Maiorana--McFarland and Desarguesian partial spread
constructions, together with their generalizations.
In Section~\ref{sec:MM}, we construct families of
 generalized Maiorana--McFarland bent functions
that lie outside the completed Maiorana--McFarland
class $\mathcal{M}^{\#}$.
In Section~\ref{sec:decom}, we investigate the
decomposition properties of generalized
$\mathcal{PS}_{ap}$ bent functions and establish
conditions under which they admit neither bent
nor semibent decompositions.
Finally, in Section~\ref{sec:conclusion}, we summarize
our main contributions and discuss several directions
for future research.

\section{Generalized Maiorana--McFarland and generalized $\mathcal{PS}_{ap}$ functions}\label{sec:MP}

In this section, we recall two classical primary constructions of bent functions and their generalizations, which are particularly relevant to our results.

\paragraph{Maiorana--McFarland ($\mathcal M$) class.}
A Boolean function $f:\F_{2^m}\times\F_{2^m}\to\F_2$ belongs to the Maiorana--McFarland class if
\[
f(x,y)=\Tr_1^m(x\pi(y))+g(y),
\]
where $\pi$ is a permutation of $\F_{2^m}$ and $g:\F_{2^m}\to\F_2$ is arbitrary. More generally, the vectorial mapping $F(x,y)=x\pi(y)+G(y)$ from $\F_{2^m}\times\F_{2^m}$ to $\F_{2^m}$ is bent whenever $\pi$ is a permutation, where $G:\F_{2^m}\to\F_{2^m}$ is arbitrary.

The following characterization of the completed class is due to Dillon.
\begin{lemma}[Dillon's criterion~\cite{dillon}]\label{lem:second-order}
	Let $n=2m$. A Boolean bent function $f:\V_n^{(2)}\to\F_2$ belongs to $\mathcal M^\#$ if and only if there exists an $m$-dimensional subspace $\mathcal U\subseteq\V_n^{(2)}$ such that $D_aD_bf$ vanishes identically for all $a,b\in\mathcal U$, where
	\begin{equation}\label{eq:second-order-derivative}
	D_aD_bf(x)=f(x)+f(x+a)+f(x+b)+f(x+a+b).
	\end{equation}
\end{lemma}
A subspace $\mathcal{U}\subseteq \V_n^{(2)}$ with this property is called an \emph{$\mathcal M$-subspace} of $f$~\cite{pp}.
An $\mathcal M$-subspace of dimension one is called
\emph{trivial}, since every one-dimensional subspace
is an $\mathcal M$-subspace of any Boolean function.
 The maximum dimension of an $\mathcal M$-subspace is the \emph{linearity index} $\operatorname{ind}(f)$. Both the number of $\mathcal M$-subspaces and the linearity index are EA-invariants. For a Boolean bent function on $\V_{2m}^{(2)}$, one has $1\le\operatorname{ind}(f)\le m$~\cite[Proposition~5.1]{Pasalic26}; in particular, $f\in\mathcal M^\#$ if and only if $\operatorname{ind}(f)=m$.

\paragraph{Generalized Maiorana--McFarland ($\widetilde{\mathcal{GMM}}$) class.}
Let $1\le k\le m$. For every $z\in\V_k^{(2)}$, let $f_z:\V_m^{(2)}\to\F_2$ be $k$-plateaued and define $\operatorname{supp}(W_{f_z})=\{b:W_{f_z}(b)\ne0\}$. Suppose these supports are pairwise disjoint. Then
\begin{equation}\label{eq:GMM}
f:\V_m^{(2)}\times\V_k^{(2)}\to\F_2,\qquad f(x,z)=f_z(x),
\end{equation}
is bent~\cite{cmpGMM}. When $k=m$, the construction specializes to the Maiorana--McFarland class.

A useful finite field form of this construction is
\begin{equation}\label{eq:genMM}
f(x,y,z)=f^{(z)}(x)+\Tr_1^k(yz),
\end{equation}
where $x\in\F_{2^n}$, $y,z\in\F_{2^k}$, and each $f^{(z)}:\F_{2^n}\to\F_2$ is bent. For each fixed $z$, the function $(x,y)\mapsto f^{(z)}(x)+\Tr_1^k(yz)$ is $k$-plateaued, with Walsh support $\F_{2^n}\times\{z\}$; hence these supports are pairwise disjoint. The resulting function $f$ is bent. Its dual is~\cite[Theorem~3]{aakm25}
\[
f^*(x,y,z)=\bigl(f^{(y)}\bigr)^*(x)+\Tr_1^k(yz),
\]
and $f^{**}=f$.

\begin{remark}
	A related generalized Maiorana--McFarland construction was introduced in~\cite[Section~4.3]{Camion92}. Let $0\le k\le n/2-1$, $r=n/2-k$, and $s=n/2+k$. One considers Boolean functions
	\begin{equation}\label{eq:MM}
	f_{\pi,h}(x,y)=\langle x,\pi(y)\rangle_r+h(y),\qquad x\in\V_r^{(2)},\ y\in\V_s^{(2)},
	\end{equation}
	where $\pi:\V_s^{(2)}\to\V_r^{(2)}$ and $h:\V_s^{(2)}\to\F_2$ are arbitrary; see also~\cite{Carlet2004}. The function in~\eqref{eq:genMM} has this form after regrouping its variables, since it is affine in $y$. When $k=0$, this recovers the classical Maiorana--McFarland form; when $k=n/2-1$, the completed class of functions of the form~\eqref{eq:MM} contains all Boolean functions on $\V_n^{(2)}$. See~\cite{Pasalic26} for further discussion.
\end{remark}

\paragraph{Desarguesian partial spread ($\mathcal{PS}_{ap}$) class.}
Let $n=2m$. A \emph{spread} of $\V_n^{(2)}$ is a collection of $m$-dimensional subspaces $U_0,U_1,\ldots,U_{2^m}$ with pairwise intersections $\{0\}$. Thus their nonzero parts partition $\V_n^{(2)}\setminus\{0\}$. A Boolean bent function arises by taking the value $0$ on the distinguished subspace $U_0$ and assigning each value $0$ and $1$ to exactly $2^{m-1}$ of the sets $U_j\setminus\{0\}$, $1\le j\le2^m$. One can similarly obtain vectorial bent functions $F:\V_n^{(2)}\to\V_m^{(2)}$ by assigning the $2^m$ values bijectively to these sets and taking a constant value on $U_0$.  

The \emph{Desarguesian spread} in $\F_{2^m}\times\F_{2^m}$ consists of
\[
U=\{(0,y):y\in\F_{2^m}\},\qquad U_s=\{(x,sx):x\in\F_{2^m}\},\quad s\in\F_{2^m}.
\]
The associated Boolean $\mathcal{PS}_{ap}$ functions have the form
\begin{equation}\label{eq:PSap}
f(x,y)=P\bigl(yx^{2^m-2}\bigr),
\end{equation}
where $P:\F_{2^m}\to\F_2$ is balanced and the convention $0^{-1}=0$ is used. In the vectorial case, $P:\F_{2^m}\to\F_{2^m}$ is a permutation. Every bent function arising from a spread has algebraic degree $m$; see~\cite[Remark~6.3.11]{dillon}.

\paragraph{Generalized $\mathcal{PS}_{ap}$ class.}
A generalization of the spread construction uses \emph{normal bent partitions}~\cite{nuwi22}. Let $K$ be even. A partition $\Omega=\{U,A_1,\ldots,A_K\}$ of $\V_{2m}^{(2)}$, where $U$ is an $m$-dimensional subspace, is a \emph{normal bent partition of depth $K$} if every Boolean function constant on $U$ and on each $A_j$, and taking each of the values $0$ and $1$ on exactly $K/2$ of the sets $A_j$, is bent. In particular, the nonzero parts of the other members of a spread, together with $U$, give a normal bent partition of depth $2^m$.
Two normal bent partitions
$\Omega=\{U,A_1,\ldots,A_K\}$ and
$\Omega'=\{U',B_1,\ldots,B_K\}$
are called \emph{equivalent} if there exists a linear
permutation $L$ of $\V_{2m}^{(2)}$ and a permutation
$\sigma$ of $\{1,\ldots,K\}$ such that
$L(U)=U'$ and $L(A_i)=B_{\sigma(i)}$
for all $1\leq i\leq K$.

Bent partitions not equivalent to those arising from spreads were introduced in~\cite{nuwi22,mp}; see also~\cite{aakm25,Anbar2025Preprint,akm22a,wfw,akm25}. 
We recall an explicit family. Let $k\mid m$ and choose an integer $e$ satisfying
\[
e\equiv2^\ell\mod{(2^k-1)},\qquad\gcd(e,2^m-1)=1.
\]
Let $\eta e\equiv1\mod{(2^m-1)}$. For $s\in\F_{2^m}$, put
\begin{equation}\label{eq:U}
U=\{(0,y):y\in\F_{2^m}\},\quad U_s=\{(x,sx^e):x\in\F_{2^m}\},\quad U_s^*=U_s\setminus\{(0,0)\},
\end{equation}
and
\begin{equation}\label{eq:Bgamma}
V=\{(x,0):x\in\F_{2^m}\},\quad V_s=\{(sx^\eta,x):x\in\F_{2^m}\},\quad V_s^*=V_s\setminus\{(0,0)\}.
\end{equation}
For $\gamma\in\F_{2^k}$, define
\begin{equation}\label{eq:Agamma}
A(\gamma)=\bigcup_{\Tr_k^m(s)=\gamma}U_s^*,\qquad B(\gamma)=\bigcup_{\Tr_k^m(s)=\gamma}V_s^*.
\end{equation}
Then
\begin{equation}\label{eq:Gamma}
\Omega_1=\{U,A(\gamma):\gamma\in\F_{2^k}\},\qquad\Omega_2=\{V,B(\gamma):\gamma\in\F_{2^k}\}
\end{equation}
are normal bent partitions of depth $2^k$. The corresponding Boolean bent functions can be written as
\[
f(x,y)=P\bigl(\Tr_k^m(yx^{-e})\bigr)+c_0(1-x^{2^m-1}),
\]
and
\[
g(x,y)=P\bigl(\Tr_k^m(xy^{-\eta})\bigr)+c_0(1-y^{2^m-1}),
\]
respectively, where $P:\F_{2^k}\to\F_2$ is balanced and $c_0\in\F_2$. Their vectorial analogues are obtained by taking $P$ to be a permutation of $\F_{2^k}$ and $c_0\in\F_{2^k}$.

When $k=m$, the partitions defined in \eqref{eq:Gamma}
are equivalent to the Desarguesian spread.
Therefore, these partitions are referred to as
\emph{generalized Desarguesian spreads}, and the functions
constructed from them are called
\emph{generalized $\mathcal{PS}_{ap}$ functions};
see~\cite{akm25,nuwi22}.

Generalized $\mathcal{PS}_{ap}$ bent functions need not belong to $\mathcal M^\#$, and examples exist outside both the classical partial spread class and $\mathcal M^\#$~\cite[Examples~5.1 and~5.2]{akm23}. The additional terms involving $c_0$ are needed to capture all EA-inequivalent functions in the generalized construction, even though they can be omitted up to EA-equivalence in the classical $\mathcal{PS}_{ap}$ case; see~\cite[Example~1]{akm25}.

\begin{remark}
The constructions presented in this section are formulated
	for Boolean functions, in accordance with the main focus
	of this paper. However, they can be viewed as special cases
	of more general constructions of $p$-ary bent functions,
	where $p$ is an arbitrary prime.
\end{remark}

	\section{Construction of \(\widetilde{\mathcal{GMM}}\) functions outside the \(\mathcal{M}^{\#}\) class} \label{sec:MM}
	
	We begin by examining the generalized Maiorana--McFarland construction given in~\cite{cmpGMM}
	over finite fields.
	Specifically, we consider functions
	\(f : \mathbb{F}_{2^n} \times \mathbb{F}_{2^k} \times \mathbb{F}_{2^k}
	\to \mathbb{F}_2\) defined by
	\begin{equation}\label{GMM}
		f(x,y,z) = f^{(z)}(x) + \Tr_1^k(yz),
	\end{equation}
	where, for each \(z \in \mathbb{F}_{2^k}\), the function
	\(f^{(z)} : \mathbb{F}_{2^n} \to \mathbb{F}_2\) is bent.

	To characterize functions lying outside the completed Maiorana--McFarland
	class, we make use of Dillon's criterion; see Lemma \ref{lem:second-order}.
	In this context, it is essential to understand the behavior of second-order
	derivatives under EA-equivalence.
	It was shown in~\cite{akkmpp26} that the existence of
	a constant second-order derivative in two linearly
	independent directions is invariant under EA-equivalence.
	For the convenience of the reader, we include the proof
	here.

	\begin{lemma}\label{lem:second}
		Let \(f \colon \V_n^{(2)} \to \F_2\) be a Boolean function, and let
		\(a,b \in \V_n^{(2)}\) be linearly independent. Let
		\(L \colon \V_n^{(2)} \to \V_n^{(2)}\) be a linear permutation, and let
		\(c,d \in \V_n^{(2)}\) and \(e \in \F_2\).
		Then \(D_a D_b f(x)\) is identically equal to \(0\) (respectively, identically equal to $1$)
		for all \(x \in \V_n^{(2)}\) if and only if
		\begin{equation}\label{eq:second}
			D_{L^{-1}(a)} D_{L^{-1}(b)} \left(\ f\bigl(L(x)+c\bigr)
			+ \langle d,x\rangle_n + e \ \right) 
		\end{equation}
		is also identically \(0\) (respectively, identically \(1\)) for all
		\(x \in \V_n^{(2)}\).
	\end{lemma}
	
	\begin{proof}
		Since any affine term \(\langle d,x\rangle_n + e\) has vanishing second-order
		derivatives, we may assume without loss of generality that
		\(g(x) = f(L(x)+c)\). By the linearity of \(L\), it follows that
		\begin{align*}
			& D_{L^{-1}(a)} D_{L^{-1}(b)} g(x) \\
			&\quad
			= g\bigl(x + L^{-1}(a) + L^{-1}(b)\bigr)
			+ g\bigl(x + L^{-1}(a)\bigr)
			+ g\bigl(x + L^{-1}(b)\bigr)
			+ g(x) \\
			& \quad
			= f\bigl(L(x + L^{-1}(a) + L^{-1}(b)) + c\bigr) + f\bigl(L(x + L^{-1}(a)) + c\bigr) \\
			&\qquad \qquad \qquad
			+ f\bigl(L(x + L^{-1}(b)) + c\bigr)
			+ f\bigl(L(x) + c\bigr) \\
			& \quad
			= f\bigl(L(x) + a + b + c\bigr)
			+ f\bigl(L(x) + a + c\bigr)
			+ f\bigl(L(x) + b + c\bigr)
			+ f\bigl(L(x) + c\bigr).
		\end{align*}
		Setting \(y = L(x) + c\), we obtain
		\[
		D_{L^{-1}(a)} D_{L^{-1}(b)} g(x)
		= f(y+a+b) + f(y+a) + f(y+b) + f(y)
		= D_a D_b f(y).
		\]
		Since the map \(x \mapsto L(x) + c\) is a permutation of \(\V_n^{(2)}\), the
		assertion follows.
	\end{proof}
	
	\begin{remark}\label{rem:fcircL}
		Let \(f \colon \V_n^{(2)} \to \F_2\) be a Boolean function, and let
		\(L \colon \V_n^{(2)} \to \V_n^{(2)}\) be a linear permutation. Then, by Lemma~\ref{lem:second},
		\(\mathcal{W} \subseteq \V_n^{(2)}\) is an \(\mathcal{M}\)-subspace of \(f\) if and only if
		\(L^{-1}(\mathcal{W})\) is an \(\mathcal{M}\)-subspace of \(f \circ L\).
	\end{remark}

	We now analyze the second-order derivative of the function
	\(f : \mathbb{F}_{2^n} \times \mathbb{F}_{2^k} \times \mathbb{F}_{2^k}
	\to \mathbb{F}_2\) defined in Equation~\eqref{GMM}.
	Recall that
	\[
	f(x,y,z) = f^{(z)}(x) + \Tr_1^k(yz).
	\]
	Let
	\(\nu_1 = (u_1,v_1,w_1)\) and
	\(\nu_2 = (u_2,v_2,w_2)\)
	be two distinct nonzero elements of
	\(\mathbb{F}_{2^n} \times \mathbb{F}_{2^k} \times \mathbb{F}_{2^k}\). Then
	\begin{align}\label{eq:secondder}
		&D_{\nu_1}D_{\nu_2}f(x,y,z)
		= f(x,y,z)
		+ f(x+u_1,y+v_1,z+w_1)
		+ f(x+u_2,y+v_2,z+w_2)\\ \nonumber
		&\qquad\qquad 
		+ f(x+u_1+u_2,y+v_1+v_2,z+w_1+w_2)\\[0.5em] \nonumber
		&=f^{(z)}(x)
		+ f^{(z+w_1)}(x+u_1)
		+f^{(z+w_2)}(x+u_2)
		+ f^{(z+w_1+w_2)}(x+u_1+u_2)
		+ \Tr_1^k\!\bigl(w_1v_2+w_2v_1\bigr).
	\end{align}
	
	\begin{remark}\label{rem:Msub}
		In the case \(w_1 = w_2 = 0\), Equation~\eqref{eq:secondder} yields
		\begin{equation}\label{eq:Dfy}
			D_{\nu_1} D_{\nu_2} f(x,y,z)
			=
			f^{(z)}(x)
			+ f^{(z)}(x+u_1)
			+ f^{(z)}(x+u_2)
			+ f^{(z)}(x+u_1+u_2).
		\end{equation}
		In particular, $D_{\nu_1} D_{\nu_2} f(x,y,z)
		=
		D_{u_1} D_{u_2} f^{(z)}(x)$,
		and hence \(D_{\nu_1} D_{\nu_2} f(x,y,z)\) is independent of \(y\).
		Consequently, if the functions \(f^{(z)}\) are \(\mathcal{M}\) functions
		sharing a common \(\mathcal{M}\)-subspace of dimension \(n/2\), then, by
		Dillon's criterion, the function \(f(x,y,z)\) belongs to \(\mathcal{M}^{\#}\).

		Indeed, it suffices to note that \(f(x,y,z)\) admits an \(\mathcal{M}\)-subspace
		of dimension \(k + n/2\).
		Let \(\mathcal{V}\) denote the common \(\mathcal{M}\)-subspace of the functions
		\(f^{(z)}\), with \(\dim \mathcal{V} = n/2\).
		Then, by Equation~\eqref{eq:Dfy}, the set $\mathcal{V}  \times \mathbb{F}_{2^k} \times \{0\}$
		forms an \(\mathcal{M}\)-subspace of \(f\) of dimension \(k + n/2\), which
		establishes the claim.
	\end{remark}

	Remark~\ref{rem:Msub} indicates that, in order to construct a \(\widetilde{\mathcal{GMM}}\) function from \(\mathcal{M}\) functions that does
	not belong to the completed Maiorana--McFarland class, it is necessary to
	combine \(\mathcal{M}\) functions that do not share a common \(\mathcal{M}\)-subspace of
	dimension \(n/2\).
	
	Motivated by this observation, we consider the construction of \(\widetilde{\mathcal{GMM}}\) functions
	from \(\mathcal{M}\) functions that share only {trivial} \(\mathcal{M}\)-subspaces, i.e.,
	subspaces of dimension one.
	Our approach begins with an \(\mathcal{M}\) function on \(\V_n^{(2)}\) admitting a unique
	\(\mathcal{M}\)-subspace \(\mathcal{W}\) of dimension \(n/2\), with the
	additional property that every nontrivial \(\mathcal{M}\)-subspace is contained in
	\(\mathcal{W}\).
	By applying a suitably chosen linear permutation to this function and
	embedding it into a higher-dimensional \(\widetilde{\mathcal{GMM}}\) framework, we obtain a \(\widetilde{\mathcal{GMM}}\) function
	that admits only \(\mathcal{M}\)-subspaces of dimension strictly smaller than
	those allowed for \(\mathcal{M}\) functions.
	
	The following proposition formalizes this construction.
	
	\begin{proposition}\label{thm:main}
		Let \(h\) be an \(\mathcal{M}\) function on \(\V_n^{(2)}\) 
		having a unique 
		\(\mathcal{M}\)-subspace \(\mathcal{W}\) of dimension \(n/2\), such that if $\mathcal{U}$ is a nontrivial
		\(\mathcal{M}\)-subspace of $h$, then $\mathcal{U}\subseteq \mathcal{W}$. Let  \(\widetilde{\mathcal{W}}\) be a complementary subspace of \(\V_n^{(2)}\) satisfying 
		\(\V_n^{(2)} =\mathcal{W} \oplus \widetilde{\mathcal{W}}\) (i.e., $\mathcal{W} \cap \widetilde{\mathcal{W}}=\{0\}$ and \(\V_n^{(2)} =\mathcal{W} + \widetilde{\mathcal{W}}\)), and let \(L\) be a linear permutation of \(\V_n^{(2)}\) with 
		\({L}( \widetilde{\mathcal{W}} ) = \mathcal{W}\).
		
		For each \(z\in\mathbb{F}_{2^k}\), define  
		\begin{align}\label{eq:g}
			f^{(z)}(x) =
			\begin{cases}
				h(x), & \text{if } z \in V,\\[6pt]
				(h\circ L)(x), & \text{if } z \not\in V,
			\end{cases}
		\end{align}
		where $	V = \bigl\{\, z \in \mathbb{F}_{2^k} : \Tr_1^k(z)=0 \,\bigr\}$.

		With this choice of \(f^{(z)}\) in \eqref{eq:g}, if \(n > 2k+4\), 
		then the function $	f:\V_n^{(2)} \times \mathbb{F}_{2^k} \times \mathbb{F}_{2^k} \to \mathbb{F}_{2}$, defined by $f(x,y,z) = f^{(z)}(x) + \Tr_1^k\!\bigl(y z\bigr)$, is not (equivalent to) an \(\mathcal{M}\) function. 
		Equivalently, if \(\mathcal{U}\) is an \(\mathcal{M}\)-subspace of \(f\), then 
		\(\dim(\mathcal{U}) < n/2 + k\).
	\end{proposition}

	\begin{proof}
	The main idea is to exploit the fact that the
			$\mathcal{M}$-subspaces of $h$ and $h\circ L$
			are contained in the complementary subspaces
			$\mathcal{W}$ and $\widetilde{\mathcal{W}}$,
			respectively.
			We proceed in two steps.
			First, we show that the subspace of an arbitrary
			$\mathcal{M}$-subspace $\mathcal{U}$ of $f$
			consisting of vectors whose last components
			have trace zero has dimension at most $2k+1$.
			We then use the linearity of the trace to conclude
			that 
			$\dim(\mathcal{U})\leq 2k+2$.
			Since $n>2k+4$, this gives
			\[
			\dim(\mathcal{U})<\frac{n}{2}+k,
			\]
			as required.
			
		We now give the details.

		We recall that, by Remark~\ref{rem:fcircL}, the subspace
		\(L^{-1}(\mathcal{W}) = \widetilde{\mathcal{W}}\) is an \(\mathcal{M}\)-subspace of
		\(h \circ L\) of dimension \(n/2\).
		Moreover, by assumption, $\widetilde{\mathcal{W}}$ is the unique
		\(\mathcal{M}\)-subspace with the property that every nontrivial
		\(\mathcal{M}\)-subspace of \(h\circ L\) is contained in  $\widetilde{\mathcal{W}}$.
		Equivalently, for {distinct} nonzero
		\({a},{b}\in \V_n^{(2)}\), the equality $D_{{a}}D_{{b}} (h \circ L) (x)=0$ holds for all $x\in \V_n^{(2)}$
		if and only if \({a},{b}\in \widetilde{\mathcal{W}}\).

		\medskip
		\noindent\textbf{Step 1.}
		Define
		\begin{equation}\label{eq:UV}
			\mathcal{U}_V
			=
			\bigl\{\,
			(u,v,w)\in \mathcal{U} : \Tr_1^k(w)=0
			\,\bigr\}.
		\end{equation}
		Let
		\(\nu_1=(u_1,v_1,w_1)\) and \(\nu_2=(u_2,v_2,w_2)\)
		be distinct nonzero elements of \(\mathcal{U}_V\).
		Since \(V\) is a subspace, \(w_1+w_2\in V\).
		
		Assume first that
		\(\Tr_1^k(w_1v_2+w_2v_1)=0\).
		Then for any \(z\in V\),
		Equation~\eqref{eq:secondder} yields
		\begin{align*}
			D_{\nu_1}D_{\nu_2}f(x,y,z)
			&= h(x)+h(x+u_1)+h(x+u_2)+h(x+u_1+u_2) \\
			&= D_{u_1}D_{u_2}h(x).
		\end{align*}
		Hence,
		\(D_{\nu_1}D_{\nu_2}f(x,y,z)=0\) for all \(x\in \V_n^{(2)}\)
		if and only if either
		\(u_1,u_2\) are distinct nonzero elements of the
		\(\mathcal{M}\)-subspace of \(h\), or $(u_1,u_2)\in \{(0,u),(u,0),(u,u):u\in \V_n^{(2)}\}$.

		Similarly, for \(z\notin V\), $D_{\nu_1}D_{\nu_2}f(x,y,z)
		= D_{u_1}D_{u_2}(h\circ L)(x)$,
		and vanishing occurs if and only if either
		\(u_1,u_2\) belong to the \(\mathcal{M}\)-subspace of
		\(h\circ L\), or  $(u_1,u_2)\in \{(0,u),(u,0),(u,u):u\in \V_n^{(2)}\}$.
		Since \(\mathcal{W}\cap \widetilde{\mathcal{W}}=\{0\}\),
		we must have $(u_1,u_2)\in \{(0,u),(u,0),(u,u):u\in \V_n^{(2)}\}$.
		
		Define the alternating bilinear form $B:\mathcal U_V\times\mathcal U_V\rightarrow\mathbb F_2$
	by
	\[
	B(\nu_i,\nu_j)
	=\Tr_1^k(w_i v_j+w_jv_i),
	\]
	where $\nu_i=(u_i,v_i,w_i)$ and
	$\nu_j=(u_j,v_j,w_j)$.
	Let $\mathfrak S$ be a subspace of $\mathcal U_V$
	of maximal dimension satisfying $B(\nu_i,\nu_j)=0$ for all $\nu_i,\nu_j\in\mathfrak S$.
	
	If $\mathfrak S$ contains two elements with
	nonzero first components, then these components
	must coincide. Hence, for some fixed nonzero
	$u\in\V_n^{(2)}$, the subspace $\mathfrak S$
	is contained in
	\[
	\{0,u\}\times
	\left\{
	(v,w)\in\mathbb F_{2^k}^2:
	\Tr_1^k(w)=0,\;
	\Tr_1^k(wv_i+w_iv)=0
	\text{ for all }\nu_i\in\mathfrak S
	\right\}.
	\]
	By the defining conditions, $\mathfrak S$
	is an $\mathcal M$-subspace of $f$.
	
	\smallskip

		\noindent\textbf{Case (i):}
		Suppose that $\dim(\mathfrak S)=1$.
		We claim that $\dim(\mathcal U_V)\leq 2$.
		
		Assume, to the contrary, that
		$\dim(\mathcal U_V)\geq 3$.
		Choose a nonzero vector
		$\nu_1=(u_1,v_1,w_1)\in\mathcal U_V$
		and consider the linear map $\varphi_{\nu_1}:\mathcal U_V\rightarrow\mathbb F_2$ defined by $\varphi_{\nu_1}(u,v,w)
		=\Tr_1^k(w_1v+wv_1)$.
		Since $\varphi_{\nu_1}(\nu_1)=0$ and
		\[
		\dim \left( \ker(\varphi_{\nu_1})
		\right)
		\geq\dim(\mathcal U_V)-1\geq 2,
		\]
		there exists a vector
		$\nu_2=(u_2,v_2,w_2)\in\ker(\varphi_{\nu_1})$
		that is linearly independent of $\nu_1$.
		Thus, $\Tr_1^k(w_1v_2+w_2v_1)=0$.
		By the linearity of the trace, this implies that $\Tr_1^k(w_iv_j+w_jv_i)=0$
		for all
		$\nu_i,\nu_j\in\langle\nu_1,\nu_2\rangle$.
		Hence, $\langle\nu_1,\nu_2\rangle$
		is a two-dimensional subspace of $\mathcal U_V$
		satisfying the defining condition of $\mathfrak S$.
		This contradicts the maximality of
		$\dim(\mathfrak S)=1$.
		Consequently, $\dim(\mathcal U_V)\leq 2$.

		\smallskip
		\noindent\textbf{Case (ii):}
		Suppose that $\dim(\mathfrak S)\geq 2$.
		If $\mathfrak S=\mathcal U_V$, then $\dim(\mathcal U_V)=\dim(\mathfrak S)\leq 2k$.

		Now assume that $\mathfrak S\neq\mathcal U_V$.
		Choose $\nu_1=(u_1,v_1,w_1)
		\in\mathcal U_V\setminus\mathfrak S$.
		By the maximality of $\dim(\mathfrak S)$,
		there exists
		$\nu_2=(u_2,v_2,w_2)\in\mathfrak S$
		such that $\Tr_1^k(w_1v_2+w_2v_1)=1$.
		Indeed, otherwise
		$\langle\mathfrak S,\nu_1\rangle$
		would be a subspace of $\mathcal U_V$
		of dimension greater than $\dim(\mathfrak S)$
		satisfying the defining condition of
		$\mathfrak S$, a contradiction.
		
		The defining properties of $\mathcal U_V$
		then force $u_2=u$ and
		$u_1\notin\{0,u\}$.
		Otherwise,
		\[
		h(x)+h(x+u_1)+h(x+u_2)
		+h(x+u_1+u_2)=0,
		\]
		which would imply that $D_{\nu_1}D_{\nu_2}f(x,y,z)=1$
		for every $z\in V$, contradicting
		$\nu_1,\nu_2\in\mathcal U_V$.
		
		Moreover, for every
		$\nu_i=(u_i,v_i,w_i)\in\mathfrak S$
		with $u_i=u$, we have $\Tr_1^k(w_1v_i+w_iv_1)=1$.
		Indeed, otherwise $\{u,u_1\}
		\subseteq\mathcal W\cap\widetilde{\mathcal W}$,
		contradicting  the assumption that
		$\mathcal W\cap\widetilde{\mathcal W}=\{0\}$.
		
		We next show that $\mathcal U_V$
		cannot contain three vectors with
		linearly independent first components.		
		Suppose, to the contrary, that
		there exist
		$\nu_i=(u_i,v_i,w_i)\in\mathcal U_V$,
		$i=1,2,3$, whose first components
		$u_1,u_2,u_3$ are linearly independent. 				In particular, $u_i$ and $u_j$ are nonzero and distinct for $i\neq j$, so $(u_i,u_j)\notin\{(0,u),(u,0),(u,u):u\in\V_n^{(2)}\}$. Hence, by the preceding argument, $\Tr_1^k(w_iv_j+w_jv_i)=1$ for all $i\neq j$.
				
		Since $\{u_1,u_2,u_3\}$ is linearly
		independent, the vectors
		\[
		\nu_1+\nu_2
		=(u_1+u_2,v_1+v_2,w_1+w_2)
		\]
		and $\nu_3$ are linearly independent
		elements of $\mathcal U_V$.
		Moreover,
		\[
		\Tr_1^k\bigl(
		(w_1+w_2)v_3+w_3(v_1+v_2)
		\bigr)
		=1+1=0.
		\]
		By the preceding argument, this implies that
		\[
		(u_1+u_2,u_3)
		\in
		\{(0,u),(u,0),(u,u):
		u\in\V_n^{(2)}\}.
		\]
		Since $u_1+u_2\neq 0$ and $u_3\neq 0$,
		we must have $u_1+u_2=u_3$,
		contradicting the linear independence
		of $u_1,u_2,u_3$.
		
		Thus, the first components of the vectors
		in $\mathcal U_V$ span a subspace of
		dimension at most two. Consequently,
		\[
		\mathcal U_V
		\subseteq
		\langle u_1,u_2\rangle
		\times
		\left\{
		(v,w)\in\mathbb F_{2^k}^2:
		\Tr_1^k(w)=0
		\right\},
		\]
		and therefore $\dim(\mathcal U_V)\leq 2k+1$.
		
		\medskip
		\noindent\textbf{Step 2:}  
		We now consider the set $\bigl\{\, (u,v,w) \in \mathcal{U} : \Tr_1^k(w)=1 \,\bigr\}$. 
		Equivalently, for a fixed $(u_1,v_1,w_1)\in\mathcal{U} $ with $\Tr_1^k(w_1)=1$, we can consider  
		\[
		\bigl\{\, (u_1,v_1,w_1),\; (u_1+u,\;v_1+v,\;w_1+w) \, : \, (u,v,w)\in \mathcal{U}, \, \Tr_1^k(w)=1 \,\bigr\},
		\]
		as both sets span the same subspace.
		Since \(\Tr_1^k(w_1)=\Tr_1^k(w)=1\), we have \(\Tr_1^k(w_1+w)=0\).  
		In particular,  
		\((u_1+u,v_1+v,w_1+w)\in \mathcal{U}_V\), 
		where \(\mathcal{U}_V\) is defined in Equation \eqref{eq:UV}.  
		Hence, \(\mathcal{U}_V\) and \((u_1,v_1,w_1)\) generate the whole space \(\mathcal{U}\), which implies that 
		\(\dim (\mathcal{U})\leq 2k+2\). Then our assumption $n > 2k + 4$ implies that $\dim(\mathcal U) < \tfrac{n}{2} + k$,
		which yields the desired conclusion.
	\end{proof}

	In Proposition~\ref{thm:main}, in order to construct functions in the \(\widetilde{\mathcal{GMM}}\) class that do not belong to the completed Maiorana--McFarland class, we start from Maiorana--McFarland functions.
	Specifically, for
	\(\V_n^{(2)} \cong \mathbb{F}_{2^m} \times \mathbb{F}_{2^m}\) with \(n = 2m\),
	we consider the function
	\(h : \mathbb{F}_{2^m} \times \mathbb{F}_{2^m} \to \mathbb{F}_2\) defined by $h(x_1,x_2) = \Tr_{1}^{m}\!\bigl(x_1\,\pi(x_2)\bigr)$,
	where \(\pi\) is a permutation of \(\mathbb{F}_{2^m}\).
	For $a=(a_{1},a_{2}), b=(b_{1},b_{2}) \in \mathbb{F}_{2^m}\times\mathbb{F}_{2^m}$, 
	\begin{align} \label{eq:der}
		D_{a}D_{b}h(x_1,x_2) 
		&= \Tr_{1}^{m}\!\Bigl(
		x_1\,\pi(x_2)
		+\,(x_1+a_{1})\,\pi(x_2+a_{2}) \nonumber \\[4pt]
		&\quad+\,(x_1+b_{1})\,\pi(x_2+b_{2})
		+\,(x_1+a_{1}+b_{1})\,\pi(x_2+a_{2}+b_{2})
		\Bigr) \nonumber \\[6pt]
		&= \Tr_{1}^{m}\!\Bigl(
		x_1\bigl(\pi(x_2)+\pi(x_2+a_{2})+\pi(x_2+b_{2})+\pi(x_2+a_{2}+b_{2})\bigr)
		\Bigr) \nonumber \\[-1pt]
		&\quad+\;\Tr_{1}^{m}\!\Bigl(
		a_{1}\,\pi(x_2+a_{2})
		+\;b_{1}\,\pi(x_2+b_{2})
		+\;(a_{1}+b_{1})\,\pi(x_2+a_{2}+b_{2})
		\Bigr).
	\end{align}
	Recall that $ \mathcal{W} =\mathbb{F}_{2^m} \times \{0\}$
	is an $\mathcal{M}$-subspace of $h$ of dimension $m=n/2$, called the canonical $\mathcal{M}$-subspace.  
	Our goal is to determine examples of $\pi$ ensuring that $h$ has this unique $\mathcal{M}$-subspace $\mathcal{W}$ of dimension $n/2$ such that, if $\mathcal{U}$ is a nontrivial $\mathcal{M}$-subspace of $h$, then necessarily $\mathcal{U} \subseteq \mathcal{W}$. Equivalently, for distinct nonzero $a, b \in \mathbb{F}_{2^m}\times\mathbb{F}_{2^m}$,
	\[
	D_{a}D_{b}\, h(x_1,x_2)=0 
	\quad\text{for all } (x_1,x_2)\in\mathbb{F}_{2^m}\times\mathbb{F}_{2^m}
	\quad\iff\quad 
	a,b\in \mathcal{W}.
	\]
	From Equation~\eqref{eq:der}, we see that $D_{a}D_{b}\, h(x_1,x_2)=0 $ for all $(x_1,x_2)\in\mathbb{F}_{2^m}\times\mathbb{F}_{2^m}$ if and only if, for every $x_2\in \mathbb{F}_{2^m}$,
	\begin{align*}
		\pi(x_2)+\pi(x_2+a_{2})+\pi(x_2+b_{2})+\pi(x_2+a_{2}+b_{2}) &= 0, \\[4pt] 
		\Tr_{1}^{m}\!\Bigl(
		a_{1}\,\pi(x_2+a_{2})
		+\;b_{1}\,\pi(x_2+b_{2})
		+\;(a_{1}+b_{1})\,\pi(x_2+a_{2}+b_{2})
		\Bigr) &= 0.
	\end{align*}
	Hence, we obtain the following result.
	\begin{corollary}\label{cor:char}
		Let
		\(h:\mathbb{F}_{2^m}\times \mathbb{F}_{2^m}\to \mathbb{F}_{2}\)
		be an \(\mathcal{M}\) function defined by $h(x_1,x_2)=\Tr_{1}^{m}\!\bigl(x_1\,\pi(x_2)\bigr)$,
		where \(\pi\) is a permutation of \(\mathbb{F}_{2^m}\).
		Then $ \mathcal{W} =\mathbb{F}_{2^m} \times \{0\}$
		is the unique \(\mathcal{M}\)-subspace of \(h\) of dimension \(m\) that contains
		all nontrivial \(\mathcal{M}\)-subspaces of \(h\) if and only if the permutation
		\(\pi\) satisfies the following property.
		
		\medskip
		\noindent\textbf{Property ($P_3$).}
		Let $a=(a_{1},a_{2}), b=(b_{1},b_{2}) \in \mathbb{F}_{2^m}\times\mathbb{F}_{2^m}$. Then the two identities
		\begin{align}
            \pi(x)+\pi(x+a_{2})+\pi(x+b_{2})+\pi(x+a_{2}+b_{2})
			&= 0, \label{eq:pi1}\\[6pt]
			\Tr_{1}^{m}\!\Bigl(
			a_{1}\,\pi(x+a_{2})
			+\;b_{1}\,\pi(x+b_{2})
			+\;(a_{1}+b_{1})\,\pi(x+a_{2}+b_{2})
			\Bigr)
			&= 0, \label{eq:pi2}
		\end{align}
        hold for all \(x \in \mathbb{F}_{2^m}\)
		if and only if one of the following holds:
		\[
		(a_1,a_2)=(0,0), \quad
		(b_1,b_2)=(0,0), \quad
		(a_1,a_2)=(b_1,b_2), \quad
		\text{or} \quad a_2=b_2=0.
		\]
	\end{corollary}

	\begin{remark}\label{rem:property}
		In \cite{PaPoKuZ}, a permutation $\pi$ of $\mathbb{F}_{2^{m}}$ is said to satisfy
		\emph{Property~($P_1$)} if $D_a D_b \pi \neq 0$ for all linearly independent
		$a,b \in \mathbb{F}_{2^{m}}$. Such permutations are used to construct Boolean
	 \(\mathcal{M}\) functions of the form $f(x,y) = \Tr_{1}^{m}\!\bigl(x\,\pi(y)\bigr) + h(y)$
		on $\mathbb{F}_{2^{m}} \times \mathbb{F}_{2^{m}}$ that possess a unique
		$m$-dimensional $\mathcal{M}$-subspace, namely the canonical one
		$\mathbb{F}_{2^{m}} \times \{0\}$.

		However, it has been observed that Property~($P_1$) is not a necessary condition
		for such \(\mathcal{M}\) functions to have a unique $m$-dimensional $\mathcal{M}$-subspace.
		Therefore, another property, \emph{Property~($P_2$)}, was introduced for a
		permutation $\pi$, providing a necessary and sufficient condition for $f(x,y) = \Tr_{1}^{m}\!\bigl(x\,\pi(y)\bigr)$
		to have a unique $m$-dimensional $\mathcal{M}$-subspace (see
		\cite[Proposition~14]{PaPoKuZ}).
		
	For brevity, we do not state the technical definition of Property~($P_2$) here
		and refer the reader to \cite[Definition~13]{PaPoKuZ}. It is known that
		Property~($P_1$) implies Property~($P_2$), but not vice versa (see
		\cite{Kudin26}). By Corollary~\ref{cor:char}, we conclude that
		Property~($P_3$) implies Property~($P_1$), which in turn
implies Property~($P_2$). Since Property~($P_2$)
does not imply Property~($P_1$), it cannot imply
Property~($P_3$) either. 

Furthermore, \cite[Remark~24]{PaPoKuZ} provides an example of a
	permutation satisfying ($P_1$) for which the associated
	Maiorana--McFarland bent function admits a
	two-dimensional $\mathcal{M}$-subspace that is not
	contained in the canonical $\mathcal{M}$-subspace.
	Consequently, this permutation does not satisfy ($P_3$),
showing that ($P_1$) does not imply ($P_3$) either.
\end{remark}	

	\begin{theorem}\label{thm:ex}
		Let \(m,k\) be positive integers, and let \(\pi\) be a permutation of
		\(\mathbb{F}_{2^m}\) satisfying \textbf{Property~($P_3$)}
		of Corollary~\ref{cor:char}.
		For each \(z \in \mathbb{F}_{2^k}\), define the Boolean function
		\(f^{(z)} : \mathbb{F}_{2^m} \times \mathbb{F}_{2^m} \to \mathbb{F}_2\) by
		\[
		f^{(z)}(x_1,x_2)=
		\begin{cases}
		\Tr_1^m\!\bigl(x_1\,\pi(x_2)\bigr), & \text{if } \Tr_1^k(z)=0, \\[6pt]
		\Tr_1^m\!\bigl(x_2\,\pi(x_1)\bigr), & \text{if } \Tr_1^k(z)=1.
		\end{cases}
		\]
		If \(m > k+2\), then the function
		\[
		f :
		\mathbb{F}_{2^m} \times \mathbb{F}_{2^m} \times
		\mathbb{F}_{2^k} \times \mathbb{F}_{2^k}
		\rightarrow \mathbb{F}_2,
		\qquad
		f(x_1,x_2,y,z)
		= f^{(z)}(x_1,x_2) + \Tr_1^k(yz),
		\]
		is a \(\widetilde{\mathcal{GMM}}\) function that does not belong to the \(\mathcal{M}^{\#}\) class.
	\end{theorem}

	\begin{proof}
		By Corollary~\ref{cor:char}, the function $h(x_1,x_2)=\Tr_{1}^{m}\!\bigl(x_1\,\pi(x_2)\bigr)$
		admits a unique \(\mathcal{M}\)-subspace containing all nontrivial
		\(\mathcal{M}\)-subspaces of \(h\), namely $ \mathcal{W}= \mathbb{F}_{2^m}\times \{0\}$.
		Set \(\widetilde{\mathcal{W}}=\{0\}\times \mathbb{F}_{2^m}\), and let
		\(L:\mathbb{F}_{2^m}\times\mathbb{F}_{2^m}\to
		\mathbb{F}_{2^m}\times\mathbb{F}_{2^m}\)
		be the linear permutation defined by
		\(L(x_1,x_2)=(x_2,x_1)\). Note that $\mathbb{F}_{2^m}\times\mathbb{F}_{2^m}
		=\mathcal{W} \oplus \widetilde{\mathcal{W}}$ and $L(\widetilde{\mathcal{W}} )=\mathcal{W} $.
		Then the conclusion follows directly from Proposition~\ref{thm:main}.
	\end{proof}
	
	We begin with an auxiliary result required to exhibit a permutation
	\(\pi\) satisfying \textbf{Property~($P_3$)} of
	Corollary~\ref{cor:char}.

	\begin{lemma}	\label{lem:trace}
		Let $m \geq 3$ 
		and let $c,d$ be nonzero elements of $\mathbb{F}_{2^m}$.  
		Then $\Tr_{1}^{m}\!\left(d\left(\tfrac{1}{x}+\tfrac{1}{x+c}\right)\right) $
		does not vanish identically on
		\(\mathbb{F}_{2^m}\setminus\{0,c\}\).
	\end{lemma}
		
	\begin{proof}
	Suppose, to the contrary, that
		\[
		\Tr_1^m\left(
		d\left(\frac{1}{x}+\frac{1}{x+c}\right)
		\right)=0
		\]
		for every $x\in\mathbb{F}_{2^m}\setminus\{0,c\}$.
		Since
		\[
		\frac{1}{x}+\frac{1}{x+c}
		=\frac{c}{x(x+c)},
		\]
		the substitution $u=x/c$ gives
		\[
		\Tr_1^m\left(
		\frac{d/c}{u^2+u}
		\right)=0
		\]
		for every $u\in\mathbb{F}_{2^m}\setminus\{0,1\}$.
		Recall that the image of the map
		$u\mapsto u^2+u$ is precisely $\mathcal{Z}=  \{z\in\F_{2^m}:\Tr_1^m(z)=0\}$.
		Thus, setting $a=d/c\ne 0$, we obtain $\Tr_1^m\left(\tfrac{a}{z}\right)=0$
		for every nonzero $z\in\mathbb{F}_{2^m}$ satisfying
		$\Tr_1^m(z)=0$.
		
		Let $\chi(t)=(-1)^{\Tr_1^m(t)}$ be the
		canonical additive character of $\F_{2^m}$.
		It follows that
		\[
		\sum_{\substack{z\in\mathbb{F}_{2^m}\setminus\{0\}\\
				\Tr_1^m(z)=0}}
		\chi\left(\frac{a}{z}\right)
		=2^{m-1}-1.
		\]
		On the other hand, using the identity $\delta_{\mathcal{Z}} (z)
		=
		\tfrac{1+\chi(z)}{2}$,
		where $\delta_{\mathcal{Z}}$ denotes
		the indicator function of the set
		$\mathcal{Z}$, we obtain the following equalities:
		\begin{align*}
		2^{m-1}-1
		&=
		\sum_{z\in\mathbb{F}_{2^m}\setminus\{0\}}
		\delta_{\mathcal{Z}} (z)
		\chi\left(\frac{a}{z}\right)=
		\frac12\sum_{z\in\mathbb{F}_{2^m}\setminus\{0\}}
		\left(1+\chi(z)\right)
		\chi\left(\frac{a}{z}\right)\\
		&=
		\frac12\left(
		\sum_{z\in\mathbb{F}_{2^m}\setminus\{0\}} \chi\left(\frac{a}{z}\right)
	+{\sum_{z\in\F_{2^m}\setminus\{0\}}}\chi(z)\chi\left(\frac{a}{z}\right)
		\right)\\
&		=
		\frac12\left(
		-1+
		\sum_{z\in\mathbb{F}_{2^m}\setminus\{0\}}
		\chi\left(z+\frac{a}{z}\right)
		\right).
		\end{align*}
		Hence, we have
		\[
		\sum_{z\in\mathbb{F}_{2^m}\setminus\{0\}}
		(-1)^{\Tr_1^m(z+a/z)}
		=2^m-1.
		\]
		By \cite[Theorem~3.4]{LW90}, the
	Kloosterman sum
	\[
	K(a):=
	\sum_{z\in\mathbb{F}_{2^m}\setminus\{0\}}
	(-1)^{\Tr_1^m(z+a/z)}
	\]
	satisfies $|K(a)|\leq 2^{m/2+1}$.
	Since $m\geq 3$, we have $2^m-1>2^{m/2+1}$,
	contradicting $K(a)=2^m-1$. Therefore, the stated trace function cannot
		vanish identically on
		$\F_{2^m}\setminus\{0,c\}$.
	\end{proof}

	\begin{lemma}\label{pro:pi}
		Let \(m \ge 3\),
		and let \(\pi\) be the permutation of \(\mathbb{F}_{2^m}\)
		defined by \(\pi(x) = x^{2^m-2}\). Then $\pi$ satisfies \textbf{Property~($P_3$)}
		of Corollary~\ref{cor:char}. That is, $\pi(x)$
		satisfies Equations~\eqref{eq:pi1}
		and~\eqref{eq:pi2} for all $x\in\mathbb{F}_{2^m}$ if and only if one of the
		following conditions holds:
		\[
		(a_1,a_2)=(0,0), \quad (b_1,b_2)=(0,0), \quad (a_1,a_2)=(b_1,b_2),
		\quad \text{or} \quad a_2=b_2=0.
		\]
	\end{lemma}

	\begin{proof}
		The sufficiency of the conditions
		$(a_1,a_2)=(0,0)$, $(b_1,b_2)=(0,0)$, $(a_1,a_2)=(b_1,b_2)$, or $a_2=b_2=0$
		is straightforward. In each case, $\pi(x)$ satisfies
		Equations~\eqref{eq:pi1} and~\eqref{eq:pi2} for all $x\in\mathbb{F}_{2^m}$.
		Therefore, it remains to establish necessity.
		
		\medskip
	
		For any nonzero \(x\in \mathbb{F}_{2^m}\), we can write
		\(\pi(x)=1/x\). Hence, for
		\(x\in \mathbb{F}_{2^m}\setminus\{0,a_2,b_2,a_2+b_2\}\),
		Equation~\eqref{eq:pi1} can be written as
		\begin{align*}
			&\pi(x)+\pi(x+a_2)+\pi(x+b_2)+\pi(x+a_2+b_2) \\[4pt]
			&= \frac{1}{x}+\frac{1}{x+a_2}+\frac{1}{x+b_2}+\frac{1}{x+a_2+b_2} \\[6pt]
			&= \frac{a_2b_2(a_2+b_2)}
			{x(x+a_2)(x+b_2)(x+a_2+b_2)}.
		\end{align*}
		Consequently, if Equation~\eqref{eq:pi1} holds for all $x\in \mathbb{F}_{2^m}$,
		then necessarily \(a_2=0\) or \(b_2=0\), or \(a_2=b_2\).
		
		\medskip
		We next consider the case \(a_2=b_2=c\) for some nonzero
		\(c\in \mathbb{F}_{2^m}\).
		For \(x\in \mathbb{F}_{2^m}\setminus\{0,c\}\),
		Equation~\eqref{eq:pi2} becomes
		\begin{align}\label{eq:a1b1}
			&\Tr_{1}^{m}\!\Bigl(
			a_{1}\,\pi(x+a_{2})
			+ b_{1}\,\pi(x+b_{2})
			+ (a_{1}+b_{1})\,\pi(x+a_{2}+b_{2})
			\Bigr) \nonumber \\[4pt]
			&= \Tr_{1}^{m}\!\Bigl(
			(a_1+b_1)\bigl(\pi(x)+\pi(x+c)\bigr)
			\Bigr) \nonumber \\[6pt]
			&= \Tr_{1}^{m}\!\Bigl(
			(a_1+b_1)\Bigl(\tfrac{1}{x}+\tfrac{1}{x+c}\Bigr)
			\Bigr).
		\end{align}
		By Lemma~\ref{lem:trace}, the above trace vanishes on
		\(\mathbb{F}_{2^m}\setminus\{0,c\}\) only if \(a_1=b_1\), i.e., $(a_1,a_2)=(b_1,b_2)$.
		
		\medskip
		Finally, suppose that exactly one of \(a_2\) or \(b_2\) is zero.
		A computation analogous to~\eqref{eq:a1b1} shows that 
		\eqref{eq:pi2} cannot hold for all \(x\in\mathbb{F}_{2^m}\) 
		unless, respectively, \(a_1=0\) or \(b_1=0\).
		
		\medskip
		This completes the proof.   
	\end{proof}
	
	As a direct consequence of Theorem~\ref{thm:ex} and Lemma~\ref{pro:pi},
	we obtain a class of functions that belong to the \(\widetilde{\mathcal{GMM}}\) class but not to the \(\mathcal{M}^{\#}\) class.
	
	\begin{corollary}\label{cor:ex1}
		Let \(m,k\) be positive integers with \(m > k+2\).
		For \(z\in \mathbb{F}_{2^k}\), define the Boolean function
		\(f^{(z)}:\mathbb{F}_{2^m}\times\mathbb{F}_{2^m}\to \mathbb{F}_{2}\) by
		\[
		f^{(z)}(x_1,x_2)=
		\begin{cases}
		\Tr_1^m\!\bigl(x_1\,x_2^{2^m-2}\bigr), & \text{if } \Tr_1^k(z)=0, \\[6pt]
		\Tr_1^m\!\bigl(x_2\,x_1^{2^m-2}\bigr), & \text{if } \Tr_1^k(z)=1.
		\end{cases}
		\]
		Then the function
		\[
		f:\mathbb{F}_{2^m}\times\mathbb{F}_{2^m}\times
		\mathbb{F}_{2^k}\times\mathbb{F}_{2^k}\to \mathbb{F}_{2},
		\qquad
		f(x_1,x_2,y,z)=f^{(z)}(x_1,x_2)+\Tr_1^k(yz),
		\]
		is a \(\widetilde{\mathcal{GMM}}\) function that is not in the \(\mathcal{M}^{\#}\) class.
	\end{corollary}

    Similarly, we consider the Gold functions
	$\pi(x)=x^{2^l+1}$ introduced in \cite{gold}
	over the finite field $\mathbb{F}_{2^m}$.
	It is well known that $\pi$ is a permutation of
	$\mathbb{F}_{2^m}$ if and only if
	$\gcd(2^l+1,2^m-1)=1$.
	Using the identity $\gcd(2^l+1,2^l-1)=1$, we obtain
	\begin{align*}
	\gcd(2^l+1,2^m-1)
	=\frac{\gcd(2^{2l}-1,2^m-1)}
	{\gcd(2^l-1,2^m-1)}=\frac{2^{\gcd(2l,m)}-1}
	{2^{\gcd(l,m)}-1}.
	\end{align*}
	Hence, $\gcd(2^l+1,2^m-1)=1$ if and only if
	$\gcd(l,m)=\gcd(2l,m)$.
	In particular, $\pi$ is a permutation of
	$\mathbb{F}_{2^m}$ if $\gcd(l,m)=1$ and $m$ is odd.

	\begin{lemma}\label{lem:gold}
		Let $m$ and $l$ be positive integers such that
		$\gcd(l,m)=1$ and $m$ is odd.
		Then $\pi(x)=x^{2^l+1}$ satisfies
		\textbf{Property~($P_3$)} of Corollary~\ref{cor:char}.
		That is, $\pi(x)$ satisfies Equations~\eqref{eq:pi1}
		and~\eqref{eq:pi2} for all $x\in\mathbb{F}_{2^m}$
		if and only if one of the following conditions holds:
		\[
		(a_1,a_2)=(0,0), \quad
		(b_1,b_2)=(0,0), \quad
		(a_1,a_2)=(b_1,b_2),
		\quad\text{or}\quad a_2=b_2=0.
		\]
	\end{lemma}
	
	\begin{proof}
Since the sufficiency of the stated conditions is immediate,
		we only prove necessity.
		
		A direct computation shows that
		\[
		\pi(x)+\pi(x+a_2)+\pi(x+b_2)+\pi(x+a_2+b_2)
		=a_2b_2\bigl(a_2^{2^l-1}+b_2^{2^l-1}\bigr).
		\]
		Since $\gcd(2^l-1,2^m-1)=1$,
		Equation~\eqref{eq:pi1} holds for all
		$x\in\mathbb{F}_{2^m}$ if and only if
		$a_2=0$ or $b_2=0$, or $a_2=b_2$.

        	Next, a straightforward computation yields
		\begin{align}\label{tr:gold}
		&\Tr_1^m\!\Bigl(
		a_1\pi(x+a_2)
		+b_1\pi(x+b_2)
		+(a_1+b_1)\pi(x+a_2+b_2)
		\Bigr)\nonumber\\[4pt]
		&\quad=\Tr_1^m\!\Bigl(
		a_1a_2^{2^l+1}
		+b_1b_2^{2^l+1}
		+(a_1+b_1)(a_2+b_2)^{2^l+1}
		+(a_1b_2+b_1a_2)x^{2^l}
		+(a_1b_2^{2^l}+b_1a_2^{2^l})x
		\Bigr)\nonumber\\[4pt]
		&\quad=\Tr_1^m\!\Bigl(
		a_1a_2^{2^l+1}
		+b_1b_2^{2^l+1}
		+(a_1+b_1)(a_2+b_2)^{2^l+1}
		+\bigl(
		(a_1b_2+b_1a_2)
		+(a_1b_2^{2^l}+b_1a_2^{2^l})^{2^l}
		\bigr)x^{2^l}
		\Bigr).
		\end{align}
Therefore, Equation~\eqref{eq:pi2} holds for all
		$x\in\mathbb{F}_{2^m}$ only if the coefficient
		of $x^{2^l}$ vanishes, namely,
		\begin{equation}\label{eq:zero}
		(a_1b_2+b_1a_2)
		+(a_1b_2^{2^l}+b_1a_2^{2^l})^{2^l}=0.
		\end{equation}

		We now analyze the possible cases.
		
		\medskip

\noindent\textbf{Case (i):} $a_2=0$.
		If $b_2=0$ or $a_1=0$, then
		Equation~\eqref{tr:gold} vanishes trivially.
		Assume $b_2a_1\neq 0$.
		Then \eqref{eq:zero} reduces to $a_1b_2+a_1^{2^l}b_2^{2^{2l}}=0$,
		which is equivalent to $a_1^{2^l-1}=b_2^{-(2^{2l}-1)}$.
		Since $\gcd(2^l-1,2^m-1)=1$, this holds
		if and only if
		$a_1=b_2^{-(2^l+1)}$.
		Since $m$ is odd, substituting into
		\eqref{tr:gold}, we obtain
		\[
		\Tr_1^m\bigl(a_1b_2^{2^l+1}\bigr)
		=\Tr_1^m(1)=1,
		\]
		a contradiction.
		Hence, either $a_2=b_2=0$ or
		$(a_1,a_2)=(0,0)$.

		\medskip
        
		\noindent\textbf{Case (ii):} $b_2=0$.
		This case is symmetric to Case~(i) and implies either $a_2=b_2=0$ or $(b_1,b_2)=(0,0)$.
		
		\medskip

\noindent\textbf{Case (iii):}
		$a_2=b_2=c$ for some nonzero
		$c\in\mathbb{F}_{2^m}$.
		Assume that $a_1\neq b_1$.
		Then \eqref{eq:zero} becomes $c(a_1+b_1)+c^{2^{2l}}(a_1+b_1)^{2^l}=0$,
		which is equivalent to $\bigl(c^{2^l+1}(a_1+b_1)\bigr)^{2^l-1}=1$.
		Since $\gcd(2^l-1,2^m-1)=1$, this holds
		if and only if
		$c^{2^l+1}(a_1+b_1)=1$.
		Substituting into \eqref{tr:gold}, we obtain
		\[
		\Tr_1^m\bigl(c^{2^l+1}(a_1+b_1)\bigr)
		=\Tr_1^m(1)=1,
		\]
		a contradiction.
		Therefore, $a_1=b_1$, and hence
		$(a_1,a_2)=(b_1,b_2)$.

		\medskip
		Combining all cases completes the proof.
	\end{proof}
	
	Similarly, Theorem~\ref{thm:ex} together with Lemma~\ref{lem:gold} yields another
	class of functions that belong to the \(\widetilde{\mathcal{GMM}}\) class but not to the completed
	\(\mathcal{M}\) class.

    	\begin{corollary}\label{cor:ex2}
		Let $m$, $k$, and $l$ be positive integers such that
		$\gcd(l,m)=1$, $m$ is odd, and $m>k+2$.
		For $z\in\mathbb{F}_{2^k}$, define the Boolean function
		$f^{(z)}:\mathbb{F}_{2^m}\times\mathbb{F}_{2^m}
		\to\mathbb{F}_2$ by
		\[
		f^{(z)}(x_1,x_2)=
		\begin{cases}
		\Tr_1^m\!\bigl(x_1x_2^{2^l+1}\bigr),
		&\text{if }\Tr_1^k(z)=0,\\[6pt]
		\Tr_1^m\!\bigl(x_2x_1^{2^l+1}\bigr),
		&\text{if }\Tr_1^k(z)=1.
		\end{cases}
		\]
		Then the function
		\[
		f:\mathbb{F}_{2^m}\times\mathbb{F}_{2^m}
		\times\mathbb{F}_{2^k}\times\mathbb{F}_{2^k}
		\to\mathbb{F}_2,
		\qquad
		f(x_1,x_2,y,z)
		=f^{(z)}(x_1,x_2)+\Tr_1^k(yz),
		\]
		is a $\widetilde{\mathcal{GMM}}$ function
		that is not in
		the $\mathcal{M}^{\#}$ class.
	\end{corollary}

	\begin{remark}
	Let $\pi(x) = x^{2^l + 1}$ be a permutation of $\mathbb{F}_{2^m}$, where $\gcd(l,m) = 2$, $m = 2r$, and $r \ge 3$ is odd. 
	In \cite[Proposition 21]{PaPoKuZ}, it was shown that $\pi$ satisfies Property~($P_2$). We recall that, by Corollary~\ref{cor:char}, Property~($P_3$) implies Property~($P_2$). 
	Hence, Lemma~\ref{lem:gold} provides a complementary result to \cite[Proposition 21]{PaPoKuZ}, while Lemma~\ref{pro:pi} gives an alternative proof that the inverse function satisfies Property~($P_2$); see \cite[Propositions 3.9 and 6.13]{Kudin26}. 
	\end{remark} 
	
	We next determine the algebraic degrees of the bent functions
	constructed in Corollaries~\ref{cor:ex1} and~\ref{cor:ex2}
	and use them to investigate whether these functions belong
	to the completed $\mathcal{PS}_{ap}$ class.
	
	\begin{lemma}
		\label{lem:degree}
		Let $f_1$ and $f_2$ be the bent functions constructed in
		Corollaries~\ref{cor:ex1} and~\ref{cor:ex2}, respectively.
		Then
		\[
		\deg(f_1)=m+1
		\qquad\text{and}\qquad
		\deg(f_2)=4.
		\]
		Consequently, $f_1\notin\mathcal{PS}_{ap}^{\#}$
		when $k\geq 2$, while
		$f_2\notin\mathcal{PS}_{ap}^{\#}$ for all admissible
		parameters.
	\end{lemma}
	
	\begin{proof}
		Both functions can be written in the form
		\[
		f(x_1,x_2,y,z)
		=
		\Tr_1^k(z)\bigl(h(x_1,x_2)+g(x_1,x_2)\bigr)
		+ h(x_1,x_2)
		+\Tr_1^k(yz),
		\]
		where $h(x_1,x_2)=\Tr_1^m(x_1x_2^d)$ and $g(x_1,x_2)=\Tr_1^m(x_2x_1^d)$
		with $d=2^m-2$ for $f_1$ and $d=2^{l}+1$ for $f_2$.
		
	Since $\deg(h)=\deg(g)=1+\mathrm{wt}_2(d)$, where  $\mathrm{wt}_2(d)$ denotes the binary Hamming
		weight of $d$,
		and their highest-degree terms have distinct
		bidegrees $(1,\mathrm{wt}_2(d))$ and
		$(\mathrm{wt}_2(d),1)$, respectively, these
		terms cannot cancel. Hence, $\deg(h+g)=1+\mathrm{wt}_2(d)$.
		Moreover, $\Tr_1^k(z)$ is a nonzero linear
		function in the variables corresponding to $z$.
		It follows that $\deg(f)=2+\mathrm{wt}_2(d)$.
		Thus, using $\mathrm{wt}_2(2^m-2)=m-1$ and $\mathrm{wt}_2(2^{l}+1)=2$,
		we obtain
		\[
		\deg(f_1)=m+1,
		\qquad
		\deg(f_2)=4.
		\]
		Recall that every bent function in the completed
		$\mathcal{PS}_{ap}$ class on $2(m+k)$ variables
		has algebraic degree $m+k$.
		For $f_1$, if $k\geq 2$, then
		\[
		\deg(f_1)=m+1<m+k.
		\]
		For $f_2$, the assumptions $m>k+2$ and $m$
		odd imply $m+k\geq 6$, and hence
		\[
		\deg(f_2)=4<m+k.
		\]
		Since algebraic degree is preserved under
		EA-equivalence for Boolean functions of degree
		at least two, the assertions follow.
	\end{proof}

	\begin{remark}
	By Lemma~\ref{lem:degree}, the bent functions
		constructed in Corollaries~\ref{cor:ex1}
		and~\ref{cor:ex2} do not belong to the completed
		$\mathcal{PS}_{ap}$ class for $k\geq 2$.
		Together with the conclusions of these corollaries,
		this shows that they belong neither to
		$\mathcal{M}^{\#}$ nor to
		$\mathcal{PS}_{ap}^{\#}$.
        For the function in Corollary~\ref{cor:ex2},
		the same conclusion also holds when $k=1$.
	\end{remark}

	\section{Decomposition of the generalized $\mathcal{PS}_{ap}$ functions } \label{sec:decom}

	Let $\mathcal{S} \subseteq \V_n^{(2)}$ be a subspace of dimension $n-2$, and
	let $\mathcal{W} \subseteq \V_n^{(2)}$ be a complementary subspace such that $\V_n^{(2)} = \mathcal{S} \oplus \mathcal{W}$.
	Let $f \colon \V_n^{(2)} \to \mathbb{F}_2$ be a Boolean function.
	For each $w_i \in \mathcal{W}$, define $f_i(x) = f(x + w_i)$ for $x \in \mathcal{S}$.
	Equivalently, $f_i$ is the restriction of $f$ to the coset
	$w_i + \mathcal{S}$, viewed as a Boolean function on $\mathcal{S}$.

	The $4$-decomposition of $f$ with respect to $\mathcal{S}$ is then defined as the
	sequence $(f_1,f_2,f_3,f_4)$,
	where $\{w_1,w_2,w_3,w_4\} = \mathcal{W}$. Throughout the paper, we simply refer to this as a decomposition (with
	respect to $\mathcal{S}$).
	If all the functions $f_i$ are bent (respectively, semibent), then $f$
	is said to admit a bent (respectively, semibent) decomposition with
	respect to $\mathcal{S}$. 
	
	We remark that, by \cite[Theorem~2]{CaCha}, the four functions
	$f_1,f_2,f_3,f_4$ arising from the decomposition of a bent
	function $f$ with respect to a codimension-two subspace
	have the same extended Walsh spectrum.
    In particular, they are either all bent, all semibent, or all have
	five-valued Walsh spectra. Thus, to determine whether
	the decomposition is bent or semibent, it suffices to
	consider the restriction of $f$ to the subspace
	$\mathcal{S}$.
	Moreover, since $\mathcal S$ has codimension $2$, there exist linearly independent vectors 
	$u,v \in \V_n^{(2)}$ such that $\mathcal S = \langle u,v\rangle^\perp$.
	In other words,
	\[
	\mathcal S
	=
	\{\, x \in \V_n^{(2)} : \langle u,x\rangle_n = \langle v,x\rangle_n = 0 \,\}.
	\]
	
	The following characterization of bent and semibent decompositions, expressed in terms of second-order derivatives of the dual function, is due to Canteaut and Charpin \cite[Theorem~7]{CaCha}.

	\begin{theorem}
		\label{lem:decom}
		Let $f \colon \V_n^{(2)} \to \F_2$ be a bent function with dual $f^*$.
		Let $u,v \in \V_n^{(2)}$ be linearly independent, and let $\mathcal{S} = \langle u,v \rangle^\perp$
		be the orthogonal complement of $\langle u,v \rangle$.
		Then the following hold:
		\begin{itemize}
			\item[(i)] $f$ admits a bent decomposition with respect to $\mathcal{S}$
			if and only if $D_u D_v f^* = 1$.
			\item[(ii)] $f$ admits a semibent decomposition with respect to $\mathcal{S}$
			if and only if $D_u D_v f^* = 0$.
		\end{itemize}
	\end{theorem}
	
	From Theorem~\ref{lem:decom}, it follows that $f$ admits a semibent
	decomposition if and only if its dual $f^*$ possesses a nontrivial
	$\mathcal{M}$-subspace; equivalently, its linearity index satisfies
	$\mathrm{ind}(f^*) > 1$. In particular, any $\mathcal{M}$ or $\widetilde{\mathcal{GMM}}$ bent function
	admits a semibent decomposition. In contrast to $\mathcal{M}$ and $\widetilde{\mathcal{GMM}}$ bent
	functions, it appears that the majority of the $\mathcal{PS}_{ap}$ bent
	functions admit neither a bent nor a semibent decomposition; for
	details, we refer to \cite{akkmpp26}.
	
	In this section, we investigate the decomposition of bent functions in the
	generalized $\mathcal{PS}_{ap}$ class. 
	As in~\cite{akkmpp26}, our approach relies on curves defined over finite fields 
	and on estimates for their number of affine rational points. 
	However, due to the definition of generalized $\mathcal{PS}_{ap}$ functions, 
	the analysis requires more involved technical computations.

	Let \(m\), \(k\), and \(e\) be integers such that \(k \mid m\),
	\(e \equiv 2^{\ell} \mod{(2^{k}-1)}\), and
	\(\gcd(2^{m}-1, e) = 1\).
	Let \(\eta\) denote the multiplicative inverse of \(e\) modulo \((2^{m}-1)\), that
	is, $\eta e \equiv 1 \mod{(2^{m}-1)}$.
	In particular, we consider the generalized $\mathcal{PS}_{ap}$ bent functions defined by
	\begin{equation}\label{eq:gPSap}
		g(x,y)=
		P\!\left(\Tr_{k}^{m}\!\left(x\, y^{-\eta}\right)\right),
	\end{equation}
	where $P \colon \mathbb{F}_{2^{k}} \to \mathbb{F}_2$ is a balanced Boolean
	function.

	To apply Theorem~\ref{lem:decom}, we first determine the dual of the function \(g\)
	defined in \eqref{eq:gPSap}.
	
	\begin{lemma}
		Let \(g \colon \mathbb{F}_{2^{m}} \times \mathbb{F}_{2^{m}} \to \mathbb{F}_2\)
		be the generalized \(\mathcal{PS}_{ap}\) bent function defined in
		\eqref{eq:gPSap}. Then the dual bent function \(g^{\ast}\) is given by
		\[
		g^{\ast}(x,y)
		=
		P\!\left(\Tr^{m}_{k}\!\left(y\, x^{-e}\right)^{2^{m-\ell}}\right)
		=
		P\!\left(\Tr^{m}_{k}\!\left(\tilde y\, \tilde x^{-e}\right)\right),
		\]
		where $\tilde y=y^{2^{m-\ell}}$ and $\tilde x=x^{2^{m-\ell}}$.
	\end{lemma}
	\begin{proof}
		For \((u,v) \in \bigl(\mathbb{F}_{2^{m}} \times \mathbb{F}_{2^{m}}\bigr)
		\setminus \{(0,0)\}\), let \(\chi_{u,v}\) denote the character of
		\(\mathbb{F}_{2^{m}} \times \mathbb{F}_{2^{m}}\) defined by
		$\chi_{u,v}(x,y)
		=
		(-1)^{\Tr_1^{m}(ux + vy)}$.
		Let $\Omega_2 = \{V, B(\gamma)\; : \; \gamma\in\F_{2^k}\}$ 
		be the generalized Desarguesian spread defined in \eqref{eq:Gamma}, where \(V\) and
		\(B(\gamma)\)  are as in \eqref{eq:Bgamma} and \eqref{eq:Agamma}, respectively.
		We note that \(g(x,y)=P(\gamma)\) for all \((x,y)\in B(\gamma)\) and
		\(g(x,y)=P(0)\) for all \((x,y)\in V\), i.e., \(g\) is a bent function arising from the bent partition
		\(\Omega_2\).
		By~\cite[Proposition~10]{akm24}, for \(\gamma \in \mathbb{F}_{2^{k}}\) we have
		\[
		\chi_{u,v}\bigl(B(\gamma)\bigr)
		=\sum_{(x,y)\in B(\gamma)}(-1)^{\Tr_1^{m}(ux + vy)}=
		\begin{cases}
		2^{m} - 2^{m-k}, &
		\text{if } u \neq 0 \text{ and }
		\gamma^{2^\ell} = \Tr_k^{m}\!\left(vu^{-e}\right), \\[1mm]
		-\,2^{m-k}, & \text{otherwise},
		\end{cases}
		\] 
		and
		\[
		\chi_{u,v}(V)
		=\sum_{(x,y)\in V}(-1)^{\Tr_1^{m}(ux + vy)}
		=
		\begin{cases}
		0, & \text{if } u \neq 0, \\[1mm]
		2^{m}, & \text{otherwise}.
		\end{cases}
		\]
		We now compute the Walsh transform of \(g\):
		\begin{align*}
			W_g(u,v)
			&=
			\sum_{(x,y) \in \mathbb{F}_{2^{m}} \times \mathbb{F}_{2^{m}}}
			(-1)^{\,g(x,y) + \Tr_1^{m}(ux+vy)} \\
			&=
			\sum_{(x,y) \in \mathbb{F}_{2^{m}} \times \mathbb{F}_{2^{m}}}
			(-1)^{\,P\!\left(\Tr_{k}^{m}\!\left(x\,y^{-\eta }\right)\right)
				+ \Tr_1^{m}(ux+vy)} \\
			&=
			\sum_{\gamma \in \mathbb{F}_{2^{k}}}
			\sum_{(x,y) \in B(\gamma)}
			(-1)^{\,P(\gamma) + \Tr_1^{m}(ux+vy)}
			+
			\sum_{(x,y) \in V}
			(-1)^{\,P(0) + \Tr_1^{m}(ux+vy)} \\
			&=
			\sum_{\gamma \in \mathbb{F}_{2^{k}}}
			(-1)^{\,P(\gamma)} \chi_{u,v}\bigl(B(\gamma)\bigr)
			+
			(-1)^{\,P(0)} \chi_{u,v}(V).
		\end{align*}
		Using the character values above together with the balancedness of
		\(P \colon \mathbb{F}_{2^{k}} \to \mathbb{F}_2\), we obtain
		\[
		W_g(u,v)
		=
		\begin{cases}
		(-1)^{P(0)}\,2^{m}, & \text{if } u = 0, \\[1mm]
		(-1)^{P(\gamma)}\,2^{m},
		& \text{if } u \neq 0 \text{, where }
		\gamma^{2^\ell} = \Tr_k^{m}\!\left(v\,u^{-e}\right).
		\end{cases}
		\]
		This yields the claimed expression for the dual function, together with the identity $\gamma = \gamma^{2^m}
		= \Tr_k^{m}\!\left(v\,u^{-e}\right)^{\,2^{m-\ell}}
		=\Tr_k^{m}\!\left(v^{\,2^{m-\ell}}\,u^{-e{\,2^{m-\ell}}}\right)$.
	\end{proof}

	Note that both the relative trace map
	$\Tr_k^m:\F_{2^m}\to\F_{2^k}$ and the function
	$P:\F_{2^k}\to\F_2$ are balanced. Consequently,
	the composition $P\circ\Tr_k^m$ is a balanced
	Boolean function on $\F_{2^m}$. Hence, there exists
	a permutation $Q$ of $\F_{2^m}$ such that $P\bigl(\Tr_k^m(s)\bigr)=\Tr_1^m(Q(s))$
	for every $s\in\F_{2^m}$.

	To construct such a permutation, let $\mathcal{Z}
	=\{z\in\F_{2^m}:\Tr_k^m(z)=0\}$,
	and fix an element $r\in\F_{2^m}$ satisfying
	$\Tr_k^m(r)=1$. Every $s\in\F_{2^m}$ admits
	a unique representation $s=\alpha r+z$ for $z\in\mathcal{Z}$ and $\alpha\in\F_{2^k}$, 
	where $\alpha=\Tr_k^m(s)$. Since $P$ is balanced, there exists a permutation
	$R$ of $\F_{2^k}$ such that $\Tr_1^k(R(\alpha))=P(\alpha)$ for every $\alpha\in\F_{2^k}$.
	Define $Q:\F_{2^m}\to\F_{2^m}$ by
	\[
	Q(s)=Q(\alpha r+z)=R(\alpha)r+z.
	\]
	Since $R$ is a permutation of $\F_{2^k}$,
	the map $Q$ is a permutation of $\F_{2^m}$.
	Equivalently, $Q(s)=s+
	\bigl(R(\Tr_k^m(s))-\Tr_k^m(s)\bigr)r$.
	Using $\Tr_k^m(r)=1$, we obtain $\Tr_1^m(Q(s))=\Tr_1^k\bigl(R(\Tr_k^m(s))\bigr)=P\bigl(\Tr_k^m(s)\bigr)$; see also
	\cite[Section~5]{Anbar2025Preprint}.

	Therefore, in the remainder of this section, we let $m$, $k$, and $e$ be positive integers such that
	\[
	k\mid m,
	\qquad
	e\equiv 2^\ell \mod{(2^k-1)},
	\qquad
	\gcd(2^m-1,e)=1,
	\]
	and we consider
	the generalized
	\(\mathcal{PS}_{ap}\) bent functions from
	\(\mathbb{F}_{2^{m}} \times \mathbb{F}_{2^{m}}\) to \(\mathbb{F}_2\)
	of the form 
	\begin{align}\label{gps:new}
		f(x,y)
		=
		\Tr_1^{m}\!\left(Q\!\left(x\,y^{-\eta }\right)\right),
	\end{align}
	where $Q$ permutes the cosets of $\mathcal{Z}$.

	We remark that, in this case, the dual function of \(f\) is given by
	\begin{align}\label{gps:newdual}
		f^{\ast}(x,y)
		=
		\Tr_1^{m}\!\left(Q\!\left(\tilde y\, \tilde x^{-e}\right)\right),
	\end{align}
	where $\tilde y=y^{2^{m-\ell}}$ and $\tilde x=x^{2^{m-\ell}}$.
	Moreover, by Lemma~\ref{lem:second}, we can without loss of generality suppose that $Q(0)=0$.

	\begin{proposition}\label{mainprop}
		Let $m$, $k$, and $e$ be positive integers satisfying
		\[
		k\mid m,\qquad
		e\equiv 2^\ell \mod{(2^k-1)},
		\qquad
		\gcd(e,2^m-1)=1.
		\]
		Let $Q$ be a permutation polynomial of $\F_{2^m}$
		that induces a permutation on the cosets of
		$\mathcal{Z}$ and has odd degree satisfying $\deg(Q)\leq \tfrac{2^{m/4}-1}{3e}$.
		Let $u=(a,b)$ and $v=(c,d)$ be two linearly independent vectors in $\mathbb{F}_{2^{m}} \times \mathbb{F}_{2^{m}}$, and set $\mathcal{S} = \langle u, v \rangle^{\perp}$. 
		Let \(f \colon \mathbb{F}_{2^{m}} \times \mathbb{F}_{2^{m}} \to \mathbb{F}_2\) be the generalized
		\(\mathcal{PS}_{ap}\) bent function defined in \eqref{gps:new}.
		If $ad + bc \neq 0$, then $f$ admits neither a bent nor a semibent decomposition on $\mathcal{S}$.
	\end{proposition}

	\begin{remark} \label{rem:mainpro}
		Let
		$f\colon \mathbb{F}_{2^{m}} \times \mathbb{F}_{2^{m}} \to \mathbb{F}_2$
		be the function defined in \eqref{gps:new}, and suppose that it satisfies the
		assumptions of Proposition~\ref{mainprop}. Then, equivalently, we have:
		\begin{itemize}
			\item[(i)] The dual function $f^{\ast}$ of $f$ does not admit any
			two-dimensional $\mathcal{M}$-subspace
			$\langle (a,b),(c,d) \rangle$ with $ad+bc \neq 0$.
			
			\item[(ii)] If $f$ admits a bent or semibent decomposition on
			$\langle (a,b),(c,d) \rangle^{\perp}$, then necessarily $ad+bc=0$.
			Equivalently, $(c,d)=\lambda (a,b)$ for some
			$\lambda \in \mathbb{F}_{2^{m}} \setminus \mathbb{F}_{2}$.
		\end{itemize}
	\end{remark}
	
	Before proving Proposition~\ref{mainprop}, we require some preliminary results.
	The first lemma is analogous to the one appearing in~\cite{akkmpp26}; however, 
	we include its proof for completeness.
	
	\begin{lemma}\label{ftof}
		Let $\mathcal{S}$ be the $(2m-2)$-dimensional subspace of $\F_{2^m}\times \F_{2^m}$ defined by
		\[
		\Tr_1^m(ax+by)=0
		\quad\text{and}\quad
		\Tr_1^m(cx+dy)=0.
		\]
		Assume that $ad+bc\neq 0$. Then the function $f(x,y)=\Tr_1^m\!\bigl(Q(xy^{-\eta})\bigr)$
		admits a bent decomposition (respectively, a semibent decomposition) on $\mathcal{S}$ if and only if $D_{(1,0)}D_{(0,1)}\widehat{f}(x,y)$
		is identically equal to $1$ (respectively, identically equal to $0$) for all \((x,y) \in \F_{2^m}\times \F_{2^m}\), where
		\[
		\widehat{f}(x,y)
		=
		\Tr_1^m\!\left(
		Q\!\left(
		\bigl(b^{2^{m-\ell}} x + d^{2^{m-\ell}} y\bigr)
		\bigl(a^{2^{m-\ell}} x + c^{2^{m-\ell}} y\bigr)^{-e}
		\right)
		\right).
		\]
	\end{lemma}
	\begin{proof}
		By Theorem~\ref{lem:decom}, the function $f$ admits a bent decomposition (respectively, a semibent decomposition) on $\mathcal{S}$ if and only if the second-order derivative $D_u D_v f^*(x,y)$
		of the dual function $f^*(x,y)$ given in \eqref{gps:newdual}
		is constantly equal to $1$ (respectively, constantly equal to $0$), where $u=(a,b)$ and $v=(c,d)$.
		
		By Lemma~\ref{lem:second}, this condition is equivalent to requiring that $D_{(1,0)} D_{(0,1)} (f^* \circ L)(x,y)$
		is constantly equal to $1$ (respectively, $0$), where $L$ is a linear permutation of
		$\F_{2^m}\times\F_{2^m}$ satisfying $L(1,0)=(a,b)$ and $L(0,1)=(c,d)$.
		
		Define the linear map $L$ by $L(x,y)=(ax^{2^\ell}+cy^{2^\ell},\; bx^{2^\ell}+dy^{2^\ell})$.
		The map $L$ permutes $\F_{2^m}\times\F_{2^m}$ if and only if $ad+bc\neq 0$. Under this assumption, we indeed have
		$L(1,0)=(a,b)$ and $L(0,1)=(c,d)$. Moreover,
		\[
		\begin{aligned}
		f^*\bigl(L(x,y)\bigr)
		&=
		\Tr_1^m\!\left(
		Q\!\left(
		(bx^{2^\ell}+dy^{2^\ell})^{2^{m-\ell}}
		(ax^{2^\ell}+cy^{2^\ell})^{-2^{m-\ell}e}
		\right)
		\right) \\[4pt]
		&=
		\Tr_1^m\!\left(
		Q\!\left(
		\bigl(b^{2^{m-\ell}}x+d^{2^{m-\ell}}y\bigr)
		\bigl(a^{2^{m-\ell}}x+c^{2^{m-\ell}}y\bigr)^{-e}
		\right)
		\right) \\
		&= \widehat{f}(x,y).
		\end{aligned}
		\]
		This yields the desired conclusion.
	\end{proof}
	By Lemma~\ref{ftof}, in order to prove Proposition~\ref{mainprop}, it suffices to show that,
	for $u=(1,0)$ and $v=(0,1)$, the second-order derivative $D_uD_v\widehat f(x,y)$
	is nonconstant.
	Since $a^{2^{m-\ell}} d^{2^{m-\ell}}
	+
	b^{2^{m-\ell}} c^{2^{m-\ell}}
	=
	(ad+bc)^{2^{m-\ell}}$,
	the condition $ad+bc \neq 0$ holds if and only if $a^{2^{m-\ell}} d^{2^{m-\ell}}
	+
	b^{2^{m-\ell}} c^{2^{m-\ell}}
	\neq 0$.
	Hence, without loss of generality, we may assume $ad+bc \neq 0$ and consider
	\[
	\widetilde{f}(x,y)
	=
	\Tr_1^m\!\left(
	Q\!\left(
	(bx+dy)(ax+cy)^{-e}
	\right)
	\right).
	\]
	More precisely, we compute
	\begin{align}
		D_uD_v\widetilde f(x,y)
		&=
		\widetilde f(x+1,y+1)
		+\widetilde f(x+1,y)
		+\widetilde f(x,y+1)
		+\widetilde f(x,y) \label{eq:Duv} \\
		&=
		\Tr_1^m\Big(
		Q\big((b(x+1)+d(y+1))(a(x+1)+c(y+1))^{-e}\big) \nonumber \\
		&\qquad
		+ Q\big((b(x+1)+dy)(a(x+1)+cy)^{-e}\big) \nonumber \\
		&\qquad
		+ Q\big((bx+d(y+1))(ax+c(y+1))^{-e}\big) \nonumber \\
		&\qquad
		+ Q\big((bx+dy)(ax+cy)^{-e}\big)
		\Big). \nonumber
	\end{align}
	
	The following lemma is frequently used to characterize the circumstances under which the second-order derivative of a function cannot be constant. 

    	\begin{lemma}\label{lem:curveDuv}
		For an integer $r \geq 1$, let $\mathcal{X}_\vartheta$ be the curve over $\F_{2^m}$ defined by
		\begin{equation} \label{eq:X_theta}
		Z^2+Z
		=
		P\!\left( \frac{\alpha_1 X+\beta_1}{\kappa_1 X+\upsilon_1} \right)
		+
		\cdots
		+
		P\!\left( \frac{\alpha_r X+\beta_r}{\kappa_r X+\upsilon_r} \right)
		+\vartheta ,
		\end{equation}
		where $P$ is a polynomial of odd degree $t$.
		Assume that $\kappa_i\ne 0$ and
		\[
		\alpha_i\upsilon_i+\beta_i\kappa_i \neq 0
		\quad \text{for } i=1,\ldots, r,
		\qquad\text{and}\qquad
		\exists j \text{ s.t. }\frac{\upsilon_j}{\kappa_j} \neq \frac{\upsilon_i}{\kappa_i}
		\ \text{for all } i\neq j.
		\]
		Then the number $N_\vartheta$ of affine $\F_{2^m}$-rational points
		$(x,z)\in \mathcal{X}_\vartheta$ with
		$x\neq \upsilon_i/\kappa_i$, for $i=1, \ldots, r$,
		satisfies
		\[
		N_\vartheta
		\;\ge\;
		2^m - 1 - (rt+1)rt\,2^{\frac{m}{2}} - r(rt+2).
		\]
		In particular, if $m \ge 4$ and $rt \le 2^{m/4}-1$, then $N_\vartheta > 0$.
	\end{lemma}

    \begin{proof}
		Let $F_\vartheta=\F_{2^m}(x,z)$ be the function field of $\mathcal{X}_\vartheta$.
		We may regard $F_\vartheta$ as an extension of the rational function field
		$\F_{2^m}(x)$ defined by the equation $z^2+z=g_\vartheta(x)$,
		where
		\[
		g_\vartheta(x)
		=
		P\!\left( \frac{\alpha_1 x+\beta_1}{\kappa_1 x+\upsilon_1} \right)
		+
		\cdots
		+
		P\!\left( \frac{\alpha_r x+\beta_r}{\kappa_r x+\upsilon_r} \right)
		+\vartheta.
		\]
		Note that $F_\vartheta/\F_{2^m}(x)$ is an Artin--Schreier extension of degree~$2$;
		see \cite[Proposition~3.7.8]{sti}.
		A place $R$ of $\F_{2^m}(x)$ ramifies in $F_\vartheta$ only if the valuation
		$v_R(g_\vartheta(x))$ is negative.
		This can occur only when $R=(x=\upsilon_i/\kappa_i)$ for $i=1, \ldots, r$.
		Moreover, by the strict triangle inequality (see \cite[Lemma~1.1.11]{sti}),
		for the place $R_j$ corresponding to
		$x = \upsilon_j/\kappa_j$, where
		$\tfrac{\upsilon_j}{\kappa_j} \neq \tfrac{\upsilon_i}{\kappa_i}$ for all $i \neq j$,
		we obtain
		\[
		v_{R_j}\bigl(g_\vartheta(x)\bigr) = -\deg(P) = -t.
		\]
		Since $t$ is odd, the place $R_j$ is totally ramified in the extension
		$F_\vartheta/\F_{2^m}(x)$.
		Consequently, $F_\vartheta$ is a function field with full constant field
		$\F_{2^m}$.
		It follows that $\mathcal{X}_\vartheta$ is an absolutely irreducible curve defined
		over $\F_{2^m}$; see \cite[Corollary~3.6.8]{sti}.
		Furthermore, $\mathcal{X}_\vartheta$ has degree
		$\deg(\mathcal{X}_\vartheta)\le rt+2$.
		Applying the Hasse-Weil bound (see \cite[Theorem~9.57]{hkt}), we obtain the following
		estimate for the number $N(\mathcal{X}_\vartheta)$ of $\F_{2^m}$-rational points of
		$\mathcal{X}_\vartheta$ in the projective plane:
		\begin{align}\label{eq:N}
		N(\mathcal{X}_\vartheta)
		&\ge
		2^m+1-(\deg(\mathcal{X}_\vartheta)-1)(\deg(\mathcal{X}_\vartheta)-2)\,2^{\frac{m}{2}} \nonumber\\
		&\ge
		2^m+1-(rt+1)rt \,2^{\frac{m}{2}}.
		\end{align}
		As the highest-degree term in the defining equation of
		$\mathcal{X}_\vartheta$ is $X^{s}Z^2$ with $s\le rt$, the curve
		$\mathcal{X}_\vartheta$ has at most two $\F_{2^m}$-rational points at infinity,
		namely $(0:1:0)$ and $(1:0:0)$.
		In order to ensure that the denominators in \eqref{eq:X_theta} do not vanish,
		we must also exclude all points lying on the $r$ lines defined by
		$\kappa_i X + \upsilon_i = 0$.
		By Bézout's theorem, a line intersects $\mathcal{X}_\vartheta$ in at most
		$\deg(\mathcal{X}_\vartheta)$ points.
		Hence, by subtracting $2+r\deg(\mathcal{X}_\vartheta)$ from \eqref{eq:N},
		we obtain the desired bound on the number $N_\vartheta$ of affine
		$\F_{2^m}$-rational points.
		Moreover, whenever $m \ge 4$
		and $rt \le 2^{m/4}-1$, we obtain
		\[
		(rt+1)rt\,2^{\frac{m}{2}} + r(rt+2) + 1 < 2^{m},
		\]
		and hence $N_\vartheta > 0$.
	\end{proof}
	
	\medskip
	
	\noindent {\it \textbf{Proof of Proposition \ref{mainprop}.}}
	The proof is carried out by a case-by-case analysis.
	
	\medskip
	
	\noindent \textbf{Case (i):} $b=d$.  \\[.3em]
	Note that the condition $b=d$ implies $b\neq 0$ and $a\neq c$, since
	$ad+bc\neq 0$.
	In this case, we set $y=x$.
	Then, by \eqref{eq:Duv}, for $x\neq \tfrac{a}{a+c},\ \tfrac{c}{a+c}$,
	we obtain
	\begin{align*}
		D_uD_v\widetilde{f}(x,x)
		&=
		\Tr^m_1\!\left(
		Q\bigl(b((a+c)x+a)^{-e}\bigr)
		+
		Q\bigl(d((a+c)x+c)^{-e}\bigr)
		\right) \\
		&=
		\Tr^m_1\!\left(
		Q\!\left( \frac{b}{((a+c)x+a)^{e}} \right)
		+
		Q\!\left( \frac{b}{((a+c)x+c)^{e}} \right)
		\right).
	\end{align*}
	Since $\gcd(e,2^m-1)=1$, the map $x\mapsto x^e$ permutes $\F_{2^m}$.
	Thus, there exists $\tilde b\in\F_{2^m}$ such that $b=\tilde b^{\,e}$.
   
Set $P(X)=Q(X^e)$. We may assume, without loss
	of generality, that $e$ is odd. Indeed, if
	$e=2^h\tilde e$ for some odd integer
	$\tilde e$, then $P(X)=\left(\widetilde Q(X^{\tilde e})\right)^{2^h}$,
	where $\widetilde Q$ is obtained from $Q$ by
	applying the inverse of the $2^h$-th Frobenius
	automorphism to its coefficients.
	Since the absolute trace is invariant under
	Frobenius automorphisms, we have
	\[
	\Tr_1^m(P(X))
	=
	\Tr_1^m\!\left(\widetilde Q(X^{\tilde e})\right).
	\]
	Thus, we may replace $P(X)$ by
	$\widetilde Q(X^{\tilde e})$, which has odd
	degree whenever $\deg(Q)$ is odd.
	
	{Moreover, this replacement does not affect
	the other assumptions of Lemma~\ref{lem:curveDuv}.
	Since $\tilde e\leq e$, the degree of the
	new polynomial does not exceed that of $P$.
	Consequently, the sufficient condition
	in Lemma~\ref{lem:curveDuv} remains satisfied,
	and the lemma is applicable to the resulting
	polynomial of odd degree.

	With $P(X)=Q(X^e)$, we see that there exists
	$x\in \F_{2^m}\setminus\left\{\frac{a}{a+c},\frac{c}{a+c}\right\}$
	such that
	$D_uD_v\widetilde f(x,x)=0$, respectively $D_uD_v\widetilde f(x,x)=1$,
	if and only if the curve
	\[
	Z^2+Z
	=
	P\!\left( \frac{\tilde b}{(a+c)X+a} \right)
	+
	P\!\left( \frac{\tilde b}{(a+c)X+c} \right)
	+\vartheta
	\]
	has an affine rational point $(x,z)$ with
	$x\notin \left\{\frac{a}{a+c},\frac{c}{a+c}\right\}$,
	where $\vartheta\in\F_{2^m}$ satisfies
	$\Tr^m_1(\vartheta)=0$, respectively $\Tr^m_1(\vartheta)=1$.

	By setting in \eqref{eq:X_theta} $r=2$,
	\[
	(\alpha_1,\beta_1,\kappa_1,\upsilon_1)=(0,\tilde b,a+c,a),
	\quad \text{and} \quad
	(\alpha_2,\beta_2,\kappa_2,\upsilon_2)=(0,\tilde b,a+c,c),
	\]
	this condition is equivalent to the curve $\mathcal{X}_\vartheta$
	having an affine $\F_{2^m}$-rational point $(x,z)$ with
	$x\notin \left\{\frac{a}{a+c},\frac{c}{a+c}\right\}$.
	Note that $\alpha_1\upsilon_1+\beta_1\kappa_1\neq 0$ and $\alpha_2\upsilon_2+\beta_2\kappa_2\neq 0$
	since $b\neq 0$ and $a\neq c$.
	Moreover, $\tfrac{\upsilon_1}{\kappa_1}\neq \tfrac{\upsilon_2}{\kappa_2}$,
	as $a\neq c$.
	Therefore, the existence of such a point follows from
	Lemma~\ref{lem:curveDuv}.
	
	\medskip
	
	\noindent \textbf{Case (ii):} $bd\neq 0$ and $b\ne d$.  \\[.3em]
	Set $y=\tfrac{b}{d}x$ and $\delta=a+\tfrac{cb}{d}$.
	By the assumption $ad+bc\neq 0$, we have $\delta\neq 0$.
	Then, for $x\neq \tfrac{a+c}{\delta},\ \tfrac{a}{\delta},\ \tfrac{c}{\delta}$,
	we obtain
	\begin{align}\label{eq:case2}
		D_uD_v\widetilde{f}\!\left(x,\frac{b}{d}x\right)
		&=
		\Tr^m_1\!\left(
		Q\!\left( \frac{b+d}{(\delta x+(a+c))^{e}} \right)
		+
		Q\!\left( \frac{b}{(\delta x+a)^{e}} \right)
		+
		Q\!\left( \frac{d}{(\delta x+c)^{e}} \right)
		\right).
	\end{align}
	As the map $x\mapsto x^e$ permutes $\F_{2^m}$,
	there exist elements
	$\tilde b,\tilde d,\widetilde{b+d}\in\F_{2^m}$ such that
	\[
	b=\tilde b^{\,e},\qquad
	d=\tilde d^{\,e},\qquad
	b+d=\left( \widetilde{b+d}\right)^{\,e}.
	\]
	Defining $P(X)=Q(X^e)$, Equation \eqref{eq:case2} can be rewritten as
	\[
	D_uD_v\widetilde{f}\!\left(x,\frac{b}{d}x\right)
	=
	\Tr^m_1\!\left(
	P\!\left( \frac{\widetilde{b+d}}{\delta x+(a+c)} \right)
	+
	P\!\left( \frac{\tilde b}{\delta x+a} \right)
	+
	P\!\left( \frac{\tilde d}{\delta x+c} \right)
	\right).
	\]
	Consequently, there exists $x\in \F_{2^m}\setminus
	\left\{\tfrac{a+c}{\delta},\tfrac{a}{\delta},\tfrac{c}{\delta}\right\}$
	such that
	$D_uD_v\widetilde{f}(x,\frac{b}{d}x)=0$,
	respectively $D_uD_v\widetilde{f}(x,\frac{b}{d}x)=1$,
	if and only if the curve defined by
	\[
	Z^2+Z
	=
	P\!\left( \frac{\widetilde{b+d}}{\delta X+(a+c)} \right)
	+
	P\!\left( \frac{\tilde b}{\delta X+a} \right)
	+
	P\!\left( \frac{\tilde d}{\delta X+c} \right)
	+\vartheta
	\]
	has an affine $\F_{2^m}$-rational point $(x,z)$ with $x\notin
	\left\{\tfrac{a+c}{\delta},\tfrac{a}{\delta},\tfrac{c}{\delta}\right\}$,
	where $\vartheta\in\F_{2^m}$ satisfies
	$\Tr^m_1(\vartheta)=0$, respectively $\Tr^m_1(\vartheta)=1$.

	We now verify that the assumptions of Lemma \ref{lem:curveDuv} are satisfied. 
	To this end, in \eqref{eq:X_theta} we set $r=3$ and 
	\[
	(\alpha_1,\beta_1,\kappa_1,\upsilon_1)
	=
	(0,\widetilde{b+d},\delta,a+c), \,
	(\alpha_2,\beta_2,\kappa_2,\upsilon_2)
	=
	(0,\tilde b,\delta,a),
	\,
	(\alpha_3,\beta_3,\kappa_3,\upsilon_3)
	=
	(0,\tilde d,\delta,c).
	\]
	Note that for every $i$ we have $\alpha_i \upsilon_i + \beta_i \kappa_i
	=
	\beta_i \delta$,
	which is nonzero since $\delta \neq 0$, $bd \neq 0$, and $b \neq d$.
	Moreover, as the $\kappa_i$ are identical, the second assumption reduces to showing that at least one of the $\upsilon_i$ differs from the other two. 
	The equalities $a+c = a = c$ would hold only if $a=0$ and $c=0$, which is impossible because $ad + bc \neq 0$.
	Therefore, by  Lemma \ref{lem:curveDuv}, such a rational point always exists.
	
	\medskip
	
	\noindent \textbf{Case (iii):} $b=0$ or $d=0 $. \\[.3em]
	Note that $b$ and $d$ cannot vanish simultaneously since $ad+bc \neq 0$.
	As the case $d=0$ is analogous to the case $b=0$, we may assume without loss of generality that $b=0$.
	Then $ad \neq 0$, and hence
	\begin{align*}
		D_uD_v\widetilde{f}(x,y)
		&= \Tr^m_1\Big(
		Q\big(d(y+1)(a(x+1)+c(y+1))^{-e}\big)
		+Q\big(dy(a(x+1)+cy)^{-e}\big)  \\
		&\qquad\quad
		+Q\big(d(y+1)(ax+c(y+1))^{-e}\big)
		+Q\big(dy(ax+cy)^{-e}\big)
		\Big).
	\end{align*}
	We now set $y=0$. Then, for $x \neq  \tfrac{c}{a},\, 1+\tfrac{c}{a}$, we obtain
	\begin{align*}
		D_uD_v\widetilde{f}(x,0)
		= \Tr^m_1\left(
		Q\left( \frac{d}{(ax+(a+c))^{e}}\right)
		+Q\left( \frac{d}{(ax+c)^{e}}\right)
		\right).
	\end{align*}
	Consequently, there exists
	$x \in \mathbb{F}_{2^m} \setminus \{ \tfrac{c}{a},\, 1+\tfrac{c}{a} \}$
	such that $D_uD_v\widetilde{f}(x,0)=0$, respectively $D_uD_v\tilde{f}(x,0)=1$,
	if and only if the curve defined by
	\begin{align*}
		Z^2+Z
		=
		P\left( \frac{\tilde{d}}{aX+(a+c)} \right)
		+P\left( \frac{\tilde{d}}{aX+c} \right)
		+\vartheta
	\end{align*}
	has an affine $\F_{2^m}$-rational point $(x,z)$ with 
	$x \notin \{ \tfrac{c}{a},\, 1+\tfrac{c}{a} \}$,
	where $P(X)=Q(X^e)$, $d=\tilde{d}^e$, and
	$\vartheta \in \mathbb{F}_{2^m}$ satisfies
	$\Tr^m_1(\vartheta)=0$, respectively $\Tr^m_1(\vartheta)=1$.
	By setting in \eqref{eq:X_theta} $r=2$,
	\begin{align*}
		(\alpha_1,\beta_1,\kappa_1,\upsilon_1)=(0,\tilde{d},a,a+c),
		\quad \text{and} \quad
		({\alpha_2},{\beta_2},{\kappa_2},{\upsilon_2})=(0,\tilde{d},a,c),
	\end{align*}
	this condition is equivalent to the curve $\mathcal{X}_\vartheta$ having an affine $\F_{2^m}$-rational point $(x,z)$ with
	$x \notin \{ \tfrac{c}{a},\, 1+\tfrac{c}{a} \}$.
	Note that
	$\alpha_i\upsilon_i+\beta_i\kappa_i \neq 0$ for $i=1,2$
	since $ad \neq 0$.
	Moreover,
	$\upsilon_1/\kappa_1 \neq {\upsilon}_2/{\kappa}_2$ since $a\neq 0$.
	Therefore, the existence of such a point follows from
	Lemma~\ref{lem:curveDuv}.
	\hfill$\Box$\\[.5em]

	We now give a complete analysis of the generalized $\mathcal{PS}_{ap}$ bent function in the case $Q(x)=x$.
	That is, we consider the bent function
	\[
	f \colon \mathbb{F}_{2^{m}} \times \mathbb{F}_{2^{m}} \to \mathbb{F}_2,
	\qquad
	f(x,y)=\Tr_1^{m}\!\left(x\,y^{-\eta}\right).
	\]
	\begin{corollary}\label{cor:x}
		Let $m$, $k$, and $e$ be positive integers such that
		\[
		k \mid m, \qquad  e \equiv 2^{\ell} \mod{(2^{k}-1)},
		\qquad 
		\gcd(2^{m}-1,e)=1.
		\]
		Let $\eta$ denote the multiplicative inverse of $e$ modulo $2^{m}-1$.
		If $4 e\leq 2^{m/4}-1$, then the generalized $\mathcal{PS}_{ap}$ bent function $f(x,y)=\Tr_1^{m}\!\left(x\,y^{-\eta}\right)$ on $\mathbb{F}_{2^{m}} \times \mathbb{F}_{2^{m}}$
		satisfies the following properties:
		\begin{itemize}
			\item[(i)] The function $f$ admits a semibent decomposition on 
			$\mathcal{S}=\langle u,v\rangle^{\perp}$, for linearly independent
			$u,v \in \mathbb{F}_{2^{m}} \times \mathbb{F}_{2^{m}}$,
			if and only if 
			$u,v \in \{0\}\times \mathbb{F}_{2^{m}}$.
			
			\item[(ii)] The function $f$ does not admit any bent decomposition.
		\end{itemize}
	\end{corollary}

	\begin{proof} Let $u=(a,b)$ and $v=(c,d)$ be two linearly independent vectors in
		$\mathbb{F}_{2^{m}} \times \mathbb{F}_{2^{m}}$.
		By Proposition~\ref{mainprop}, the function $f$ does not admit any semibent or bent
		decomposition on $\mathcal{S}=\langle u,v\rangle^{\perp}$ whenever $ad+bc \neq 0$.
		Therefore, we restrict our attention to the case $ad+bc = 0$.

		By Theorem~\ref{lem:decom}, the function $f$ admits a semibent (respectively, bent)
		decomposition on $\mathcal{S}=\langle u,v\rangle^{\perp}$ if and only if its dual $f^{\ast}(x,y)=\Tr_1^{m}\!\left(y^{2^\ell}\,x^{-2^\ell e}\right)$
		satisfies
		\[
		D_u D_v f^{\ast}=0
		\quad\text{(respectively, } D_u D_v f^{\ast}=1\text{).}
		\]
		Note that $f^*(x,y)=\widetilde f(L(x,y))$,
		where $\widetilde f(x,y)=\Tr_1^m(yx^{-e})$ and $L(x,y)=(x^{2^\ell},y^{2^\ell})$.
	Since $L$ is an invertible linear map,
	Lemma~\ref{lem:second} implies that
	$D_uD_v f^*$ is identically equal to $0$
	(respectively, $1$) if and only if
	$D_{L(u)}D_{L(v)}\widetilde f$ is identically
	equal to $0$ (respectively, $1$).
	Moreover, $L$ preserves linear independence and
		the condition $ad+bc=0$.
		Therefore, we may equivalently study the
		second-order derivatives of $\widetilde f$.
		For simplicity, we relabel the transformed
		directions $L(u)$ and $L(v)$ as $u=(a,b)$
		and $v=(c,d)$, respectively.
        
		Recall that, since $u=(a,b)$ and $v=(c,d)$ are linearly independent vectors in
		$\mathbb{F}_{2^{m}} \times \mathbb{F}_{2^{m}}$, the condition $ad+bc = 0$
		is equivalent to $(c,d)=\lambda (a,b)$ for some $ \lambda \in \mathbb{F}_{2^{m}} \setminus \mathbb{F}_{2}$.
		By replacing $c=\lambda a$ and $d=\lambda b$,
		the second-order derivative $D_u D_v \widetilde{f}(x,y)$ can be expressed as
		\begin{align*}
			D_u D_v \widetilde{f}(x,y)
			=
			\Tr_1^{m}\!\Big(
			&(y+b(1+\lambda))\,(x+a(1+\lambda))^{-e}
			+(y+b)\,(x+a)^{-e} \\
			& \qquad \qquad \qquad \qquad\qquad +(y+\lambda b)\,(x+\lambda a)^{-e}
			+y\,x^{-e}
			\Big).
		\end{align*}

		We now proceed by a case-by-case analysis.
		\medskip
		
		\noindent \textbf{Case (i):}  $a=0$, and hence $b\neq 0$.  \\[.3em]
		In this case, we obtain $D_u D_v \widetilde{f}=0$ for all such choices of $u$ and $v$.
       	In particular, $D_uD_v\widetilde f=0$ for any linearly independent $u,v\in\{0\}\times\mathbb{F}_{2^m}$.
		Since 
		\[
		L^{-1}(\{0\}\times \mathbb{F}_{2^{m}})
		=
		\{0\}\times \mathbb{F}_{2^{m}},
		\]
		it follows from Lemma~\ref{lem:second}  and Theorem~\ref{lem:decom} that 
		$f(x,y)$ admits a semibent decomposition on 
		$\langle u,v\rangle^{\perp}$ for all linearly independent 
		$u,v \in \{0\}\times \mathbb{F}_{2^{m}}$.
		
		\medskip
		
		\noindent \textbf{Case (ii):} $b=0$, and hence $a\neq 0$.  \\[.3em]
		Then applying the changes of variables $x\mapsto ax$ and $y\mapsto a^{e}y$ 
		and then setting $y=1$, we obtain
		\begin{align}\label{eq1}
			g(x)
			&=
			\Tr_1^{m}\!\Big(
			(x+1+\lambda)^{-e}
			+   (x+1)^{-e}
			+ (x+\lambda )^{-e}
			+ x^{-e}
			\Big).
		\end{align}
		If $D_uD_v\widetilde f$ were constant,
		then $g(x)$ would also be constant,
		contradicting Lemma~\ref{lem:curveDuv}.
		Thus, $D_uD_v\widetilde f$ is nonconstant.
		By Lemma~\ref{lem:second}, the corresponding
		second-order derivative of $f^*$ is also
		nonconstant.
		
		Consequently, by Theorem~\ref{lem:decom},
		the function $f$ admits neither a semibent
		nor a bent decomposition on
		$\langle u,v\rangle^\perp$
		for any linearly independent
		$u,v\in\F_{2^m}\times\{0\}$.
		
		\medskip
		
		\noindent \textbf{Case (iii):} $ab \neq 0$.  \\[.3em]
		Proceeding as before, we apply the change of variables $x \mapsto ax$ and
		$y \mapsto by$, and then set $\alpha = b a^{-e}$ and $y=0$. This yields
		\begin{align}\label{eq2}
			g(x)
			&=
			\Tr_1^{m}\!\Big(
			\alpha\Big(
			(1+\lambda)\,(x+1+\lambda)^{-e}
			+ (x+1)^{-e}
			+ \lambda (x+\lambda)^{-e}
			\Big)\,
			\Big).
		\end{align}
		If $D_uD_v\widetilde f$ were constant,
		then $g(x)$ would also be constant,
		contradicting Lemma~\ref{lem:curveDuv}.
		Hence, $D_uD_v\widetilde f$ is nonconstant.
		By Lemma~\ref{lem:second}, the corresponding
		second-order derivative of $f^*$ is also
		nonconstant.
		
		Therefore, by Theorem~\ref{lem:decom},
		the function $f$ admits neither a semibent
		nor a bent decomposition on
		$\langle u,v\rangle^\perp$
		for any linearly independent
		$u=(a,b)$ and $v=(c,d)$ satisfying
		$ad+bc=0$ and $ab\neq0$.

        \medskip
		
		Combining the three cases with
		Proposition~\ref{mainprop}, we obtain
		both assertions.
	\end{proof}
	
	\begin{remark}
		Note that the generalized $\mathcal{PS}_{ap}$ bent function $f$ in
		Corollary~\ref{cor:x}, and hence its dual $f^{\ast}$, is also an $\mathcal{M}$ function.
		Consequently, $f$ and $f^{\ast}$ admit $m$-dimensional $\mathcal{M}$-subspaces,
		namely the canonical ones $\mathbb{F}_{2^{m}} \times \{0\}$ and
		$\{0\} \times \mathbb{F}_{2^{m}}$, respectively.
		It follows from Corollary~\ref{cor:x} that
		$\{0\} \times \mathbb{F}_{2^{m}}$ is the unique $m$-dimensional
		$\mathcal{M}$-subspace of $f^{\ast}$.
	\end{remark}

Corollary~\ref{cor:x} can be extended to
permutation polynomials $Q$ of $\F_{2^m}$
that induce a permutation on the cosets of $\mathcal{Z}
=\{z\in\F_{2^m}:\Tr_k^m(z)=0\}$,
provided that their degrees are sufficiently
small relative to the cardinality of $\F_{2^m}$.
In particular, combined with
	Proposition~\ref{mainprop}, we obtain the following sufficient conditions under
	which the generalized $\mathcal{PS}_{ap}$ bent function admits neither a
	semibent nor a bent decomposition. The proof proceeds analogously to that of Corollary~\ref{cor:x}; therefore, we only give a sketch.

		\begin{theorem}\label{thm:gps}
			Let $m$, $k$, and $e$ be positive integers such that
			\[
			k\mid m,
			\qquad
			e\equiv 2^\ell \mod{(2^k-1)},
			\qquad
			\gcd(2^m-1,e)=1.
			\]
			Let $\eta$ denote the multiplicative inverse of $e$
			modulo $2^m-1$, and let $\mathcal{Z}
			=\{z\in\F_{2^m}:\Tr_k^m(z)=0\}$.
			Suppose that $Q$ is a permutation polynomial of
			$\F_{2^m}$ that induces a permutation on the
			cosets of $\mathcal{Z}$, satisfies $Q(0)=0$,
			and has odd degree at most $\tfrac{2^{m/4}-1}{4e}$.
			Define the generalized $\mathcal{PS}_{ap}$ bent function 
			$f \colon \mathbb{F}_{2^{m}} \times \mathbb{F}_{2^{m}} \to \mathbb{F}_2$ by $f(x,y)= \Tr_1^{m}\!\left(Q\!\left(x\,y^{-\eta}\right)\right)$.
			Then $f$ admits neither a semibent nor a bent decomposition provided that, 
			for every $\lambda \in \mathbb{F}_{2^{m}} \setminus \mathbb{F}_2$, 
			the Boolean function
			\begin{equation}\label{eq:Qgen}
			g_\lambda(x)
			= \Tr_1^{m}\!\left(
			Q\!\left((1+\lambda)x^{-e}\right)
			+ Q\!\left(x^{-e}\right)
			+ Q\!\left(\lambda x^{-e}\right)
			\right)
			\end{equation}
			on $\mathbb{F}_{2^{m}}$ is nonconstant.
		\end{theorem}
	\noindent{\it\textbf{Sketch of the proof.}}
	Recall that, under the assumptions on $e$, the collection
		\[
		\Omega_2
		=
		\{V\}\cup\{B(\gamma):\gamma\in\F_{2^k}\},
		\]
		where $V$ and $B(\gamma)$ are defined in
		\eqref{eq:Bgamma} and \eqref{eq:Agamma},
		forms a normal bent partition of depth $2^k$.
	
		Since $Q$ is a permutation of $\F_{2^m}$
		satisfying
		\[
		\Tr_1^m(Q(z))
		=
		P\bigl(\Tr_k^m(z)\bigr),
		\qquad z\in\F_{2^m},
		\]
		for some balanced Boolean function
		$P:\F_{2^k}\to\F_2$, $f$ takes the constant value
		$P(\gamma)$ on each $B(\gamma)$.
		Moreover, $f$ vanishes on $V$, since $Q(0)=0$.
		Thus, by the defining property of the normal
		bent partition $\Omega_2$, we conclude that
		$f$ is bent.
	
	Similarly, by Proposition~\ref{mainprop}, it suffices to consider bent or semibent 
	decompositions on $\mathcal{S}=\langle u,v\rangle^{\perp}$,
	where $u=(a,b)$ and $v=(c,d)$ are linearly independent vectors in 
	$\mathbb{F}_{2^{m}} \times \mathbb{F}_{2^{m}}$ satisfying $ad+bc=0$.
	This condition implies that
	$(c,d) =\lambda (a,b)$ for some 
	$\lambda \in \mathbb{F}_{2^{m}}\setminus \mathbb{F}_{2}$.
	Moreover, by Lemma~\ref{lem:second}, it suffices to consider $\widetilde{f}(x,y)=\Tr_1^{m}(Q(y\,x^{-e}))$.

	\medskip
	
	\noindent
	\emph{Case $a=0$ (and hence $b\neq 0$).}
	Setting $y=0$ and applying the change of variables $x \mapsto \tilde{b}x$,
	where $\tilde{b}^{\,e}=b$, to $D_u D_v \widetilde{f}(x,y)$,
	we obtain the function
	$g_\lambda(x)$ in \eqref{eq:Qgen}. Hence, if $g_\lambda(x)$ is not constant, then
	$D_u D_v f^{\ast}(x,y)$ cannot be constant.
	
	\medskip
	
	\noindent
	\emph{Case $b=0$ (and hence $a\neq 0$).}
	Applying the changes of variables $x \mapsto ax$ and $y \mapsto a^{e}y$ to
	$D_u D_v \widetilde{f}(x,y)$ and then setting $y=1$, we obtain
	\[
	\Tr_1^{m}\!\Big(
	Q\!\left((x+1+\lambda)^{-e}\right)
	+ Q\!\left((x+1)^{-e}\right)
	+ Q\!\left((x+\lambda)^{-e}\right)
	+ Q\!\left(x^{-e}\right)
	\Big),
	\]
	which cannot be constant by Lemma~\ref{lem:curveDuv}. Consequently,
	$D_u D_v f^{\ast}(x,y)$ cannot be constant.
	
	\medskip
	
	\noindent
	\emph{Case $ab\neq 0$.}
	We apply the changes of variables $x \mapsto ax$ and $y \mapsto by$, and then set
	$\alpha = b a^{-e}$ and $y=0$. This yields
	\[
	\Tr_1^{m}\!\left(
	Q\!\left(\alpha (1+\lambda)(x+1+\lambda)^{-e}\right)
	+ Q\!\left(\alpha (x+1)^{-e}\right)
	+ Q\!\left(\alpha \lambda(x+\lambda)^{-e}\right)
	\right).
	\]
	Again, by Lemma~\ref{lem:curveDuv}, this expression cannot be constant, and hence
	$D_u D_v f^{\ast}(x,y)$ cannot be constant.
	
	\medskip
	
	The result now follows from Theorem~\ref{lem:decom}, which completes the proof.
	\hfill$\Box$
	
	\medskip
	
	We now apply Theorem~\ref{thm:gps} to monomial
	permutation polynomials $Q(x)=x^d$.
	Taking $k=m$, the coset-permutation condition
	is automatically satisfied. Moreover, in this
	case, the nonconstancy condition in
	Theorem~\ref{thm:gps} reduces to a simple
	condition on the exponent $d$.
	This yields the following corollary.
	
	\begin{corollary}\label{cor:generalmonomial}
		Let $m$, $e$, and $d$ be positive integers, with $d$ odd,
		such that $	e\equiv 2^j\mod{(2^m-1)}$
		for some nonnegative integer $j$, and
		\[
		\gcd(d,2^m-1)=1,
		\qquad
		d\leq\frac{2^{m/4}-1}{4e}.
		\]
		Suppose that
		\[
		(1+\lambda)^d+1+\lambda^d\neq0
		\]
		for every $\lambda\in\F_{2^m}\setminus\F_2$.
		Let $\eta$ denote the multiplicative inverse
		of $e$ modulo $2^m-1$.
		Then the generalized $\mathcal{PS}_{ap}$
		bent function
		\[
		f(x,y)=\Tr_1^m\!\left(x^d y^{-\eta d}\right)
		\]
		admits neither a semibent nor a bent decomposition.
	\end{corollary}
		
	\begin{proof} 
		We apply Theorem~\ref{thm:gps} with $k=m$
		and $Q(x)=x^d$. Since $\gcd(d,2^m-1)=1$,
		the polynomial $Q$ is a permutation of
		$\F_{2^m}$ satisfying $Q(0)=0$.
		Moreover, since $k=m$, we have
		$\mathcal{Z}=\{0\}$, and hence $Q$ induces
		a permutation on the cosets of $\mathcal{Z}$.
	
		It remains to verify that the function
		$g_\lambda$ defined in \eqref{eq:Qgen}
		is nonconstant for every
		$\lambda\in\F_{2^m}\setminus\F_2$.
		Observe that
		\begin{align*}
		g_\lambda(x)
		&=\Tr_1^m\!\left(
		Q((1+\lambda)x^{-e})
		+Q(x^{-e})
		+Q(\lambda x^{-e})
		\right)\\
		&=\Tr_1^m\!\left(
		\bigl((1+\lambda)^d+1+\lambda^d\bigr)
		x^{-ed}
		\right).
		\end{align*}
		Since $\gcd(e,2^m-1)=\gcd(d,2^m-1)=1$,
		we have $\gcd(ed,2^m-1)=1$.
		Consequently, the mapping
		$x\mapsto x^{-ed}$, with $0$ mapped to $0$,
		is a permutation of $\F_{2^m}$.
		Furthermore, by assumption, $(1+\lambda)^d+1+\lambda^d\neq0$
		for every $\lambda\in\F_{2^m}\setminus\F_2$.
		It follows that $g_\lambda$ is balanced,
		and hence nonconstant, for every such $\lambda$.

		Therefore, the result follows from
		Theorem~\ref{thm:gps}.
	\end{proof} 
	
	We next establish a connection between
Corollary~\ref{cor:generalmonomial} and
almost perfect nonlinear (APN) functions.
Recall that a function
$G:\F_{2^m}\to\F_{2^m}$ is called APN if,
for every $a\in\F_{2^m}^*$ and
$b\in\F_{2^m}$, the equation
\[
G(x+a)+G(x)=b
\]
has at most two solutions $x\in\F_{2^m}$
(see, e.g., \cite{BCCl06}).
	
	In particular, if the monomial $G(x)=x^d$
	is APN, then the equation $(x+1)^d+x^d=1$
	has exactly two solutions, namely
	$x=0$ and $x=1$. Consequently,
	\[
	(1+\lambda)^d+1+\lambda^d\neq 0
	\]
	for every $\lambda\in\F_{2^m}\setminus\F_2$.
	Thus, the nonvanishing condition in
	Corollary~\ref{cor:generalmonomial}
	is automatically satisfied by APN monomials.
	This yields the following consequence.

		\begin{corollary}\label{cor:APN} 
		Let $m$, $e$, and $d$ be positive integers,
		where $d$ is odd, such that $e\equiv 2^j\mod{(2^m-1)}$
		for some nonnegative integer $j$, and $d\leq\tfrac{2^{m/4}-1}{4e}$.
		Suppose that $G(x)=x^d$ is an APN
		permutation of $\F_{2^m}$, and let
		$\eta$ denote the multiplicative inverse
		of $e$ modulo $2^m-1$.
		Then the generalized $\mathcal{PS}_{ap}$
		bent function
		\[
		f(x,y)=
		\Tr_1^m\!\left(x^d y^{-\eta d}\right)
		\]
		admits neither semibent nor bent
		decompositions.
	\end{corollary}
	
	We illustrate Corollary~\ref{cor:APN}
	using two well-known families of APN permutation
	monomials.

\begin{example}
	Let $m$ be an odd positive integer, and let $\ell$ be
	a positive integer satisfying $\gcd(\ell,m)=1$.
	It is well known that the Gold and Kasami monomials
	\[
	G_1(x)=x^{2^\ell+1}
	\quad\text{and}\quad
	G_2(x)=x^{2^{2\ell}-2^\ell+1}
	\]
	are APN permutations of $\F_{2^m}$; see \cite[Section 11.5]{Carlet2021}.
	
Let $e$ be a positive integer satisfying $e\equiv 2^j\mod{(2^m-1)}$
	for some nonnegative integer $j$, and let $\eta$
	denote the multiplicative inverse of $e$
	modulo $2^m-1$.	If $\ell+j\leq
	\left\lfloor {m}/{4}\right\rfloor-3$,
	where $\lfloor\cdot\rfloor$ denotes the floor
	function, then the generalized $\mathcal{PS}_{ap}$
	bent function
	\[
	f_1(x,y)=
	\Tr_1^m\!\left(
	x^{2^\ell+1}
	y^{-\eta(2^\ell+1)}
	\right)
	\]
	admits neither a semibent nor a bent decomposition.
	Similarly, if $2\ell+j\leq
	\left\lfloor {m}/{4}\right\rfloor-3$,
	then the generalized $\mathcal{PS}_{ap}$
	bent function
	\[
	f_2(x,y)=
	\Tr_1^m\!\left(
	x^{2^{2\ell}-2^\ell+1}
	y^{-\eta(2^{2\ell}-2^\ell+1)}
	\right)
	\]
	admits neither a semibent nor a bent decomposition.
\end{example}

\begin{remark}
	Recall that, under the assumptions $k\mid m$, $e\equiv 2^j\mod{(2^k-1)}$ and $\gcd(e,2^m-1)=1$,
	the collection
	\[
	\Omega_2
	=
	\{V\}\cup\{B(\gamma):\gamma\in\F_{2^k}\},
	\]
	where \(V\) and \(B(\gamma)\) are defined in
	\eqref{eq:Bgamma} and \eqref{eq:Agamma},
	forms a normal bent partition of depth \(2^k\),
	referred to as a generalized Desarguesian spread.
	
	In the special case \(k=m\), the exponent \(e\)
	is congruent to a power of \(2\) modulo \(2^m-1\),
	and the same holds for its multiplicative inverse
	\(\eta\) modulo \(2^m-1\).
	Consequently, \(\Omega_2\) is equivalent to the
	classical Desarguesian spread; see
	\cite[Remark~5]{mp}.
	In particular, the bent functions arising from
	\(\Omega_2\) belong, up to equivalence, to the
	classical \(\mathcal{PS}_{ap}\) class.
	
	Thus, when \(k=m\), Corollary~\ref{cor:APN}
	extends \cite[Theorem~5.7]{akkmpp26}, although
	under a more restrictive bound on the degree
	of the permutation polynomial.
	Nevertheless, our construction produces additional families
	 of bent functions in the
	\(\mathcal{PS}_{ap}\) class that admit neither
	a bent nor a semibent decomposition.
	Our results therefore provide a partial answer
	to \cite[Open Problem~1]{akkmpp26}, which asks
	for further families of bent functions with
	these properties.
\end{remark}

	\section{Conclusion}\label{sec:conclusion}
	
	In this paper, we investigated generalizations of two
	classical constructions of Boolean bent functions,
	namely the Maiorana--McFarland and the
	$\mathcal{PS}_{ap}$ constructions, with particular
	emphasis on construction and decomposition.
	
First, we constructed new families of generalized
	Maiorana--McFarland bent functions that are not
	EA-equivalent to any function in the classical
	$\mathcal{M}$ or $\mathcal{PS}_{ap}$ classes.
	Our results therefore provide further examples
	of bent functions outside the corresponding
	completed classes.
	
We then studied the decomposition properties of
	generalized $\mathcal{PS}_{ap}$ bent functions.
	Using second-order derivatives and techniques
	from the theory of algebraic curves and function
	fields, we established sufficient conditions
	under which these functions admit neither a bent
	nor a semibent decomposition. In particular,
	our results yield further families of bent
	functions in the classical $\mathcal{PS}_{ap}$
	class with this property, thereby providing a
	partial answer to the open problem posed in
	\cite[Open Problem~1]{akkmpp26}.

Our results demonstrate the richness of bent functions
arising from generalizations of classical constructions
and provide new insights into the limitations of
their decomposition into bent or semibent functions.

	Several questions remain open.
	One direction is to construct new infinite families
	of permutation polynomials satisfying Property $(P_3)$,
	which could lead to further families of generalized
	Maiorana--McFarland bent functions outside the
	completed classical classes.
	
	Another direction concerns the decomposition properties
	of generalized $\mathcal{PS}_{ap}$ bent functions.
	Although we establish sufficient conditions for the nonexistence
of bent and semibent decompositions in the generalized
$\mathcal{PS}_{ap}$ class, with the exception of those presented
in Corollary~\ref{cor:x}, all explicit examples obtained in this
paper belong to the classical $\mathcal{PS}_{ap}$ class.
Nevertheless, these examples enrich the known families of
classical $\mathcal{PS}_{ap}$ bent functions admitting neither
bent nor semibent decompositions.
	It would therefore be particularly interesting to
	construct families of genuinely generalized
	$\mathcal{PS}_{ap}$ bent functions, lying outside
	the classical $\mathcal{PS}_{ap}$ class, that
	also admit neither bent nor semibent decompositions.

	\section*{Acknowledgements}

	The initial work on this project began during the ``Women in Numbers Europe~5 (WINE-5)'' workshop, held at the University of Split in August 2025. The authors are grateful to the University of Split and the supporting institutions for making this conference and the resulting collaboration possible. They would especially like to thank the organizers of WINE-5---Marcela Hanzer, Borka Jadrijević, Pınar Kılıçer, and Lejla Smajlović---for their dedication and hard work, which made the meeting both exceptionally fruitful and enjoyable.

The authors would also like to thank Alexandr Polujan
for his valuable comments on an earlier version of
the manuscript. 
They are also grateful to the
anonymous referees for their careful reading and
constructive comments, particularly for identifying
several inaccuracies and suggesting clarifications
that helped improve the presentation and precision
of the manuscript.
	
	This study was supported by the Scientific and Technological Research  Council of Turkey (TÜBİTAK) under Grant Number 125F396. S.~A. and N.~A. thank TÜBİTAK for its support.
	T.~K. is supported by the FWF Project~P~35138.
	B.~T. is supported by the Swiss National Science Foundation through grant no.~212865.
    E.~E.~A. is supported by the OPIGRID project partially funded by the Department of Economic Development and Infrastructures of the Basque Government under the grant agreement KK-2026/00071.

\end{document}